\documentclass[11pt]{article}
\usepackage{subfigure}
\usepackage{amssymb,amsmath,amsthm,epsfig,a4,verbatim,pstricks,ifthen,setspace,ulem,color} 
\usepackage{url}
\usepackage{float}
\allowdisplaybreaks
\definecolor{mygreen}{rgb}{0.05,0.6,0.05}
\newtheorem{thm}{\sc Theorem.}[section]
\newtheorem{lem}{\sc Lemma.}[section]

\renewcommand{\theequation}{\arabic{section}.\arabic{equation}}
\newenvironment{AMS}%
{{\upshape\bfseries AMS subject classifications. }\ignorespaces}{}
\newenvironment{keywords}{{\upshape\bfseries Key words. }\ignorespaces}{}

\newcommand{\tx}{\tilde{x}}

\newcommand{\tchi}{\tilde{\chi}}

\def\ba{\begin{align}}
\def\ea{\end{align}}

\newcommand{\dg}{\mbox{det} \, G}
\newcommand{\dgh}{\mbox{det} \, G_h}
\newcommand{\dghm}{\mbox{det} \, G_h^m}
\newcommand{\dghmph}{\mbox{det} \, G_h^{m+\frac12}}
\newcommand{\nuu}{\tilde{\nu}}
\newcommand{\tm}{\tilde{ \mathcal T}_{\mathcal{M}_h}}
\newcommand{\hmaxm}{h_{max}}
\newcommand{\tgam}{{\mathcal T}^m_h}

\def\epsilon{\varepsilon} 

\def\hat{\widehat}

\begin{document}
\title{Error analysis of a finite element scheme for parametric mean curvature flow based
on the DeTurck trick }
\author{Klaus Deckelnick  \footnotemark[3]\ \and 
        Vanessa Styles \footnotemark[4]}

\renewcommand{\thefootnote}{\fnsymbol{footnote}}
\footnotetext[3] { Otto-von-Guericke-Universit{\"a}t Magdeburg, Institut f{\"u}r Analysis und Numerik,
39106 Magdeburg, Germany}
\footnotetext[4]{Department of Mathematics, University of Sussex, Brighton, BN1 9RF, UK}

\date{}

\maketitle

\begin{center}
{\it We dedicate this paper to Charlie Elliott on the occasion of his 75th birthday.}
\end{center}

\begin{abstract}
The paper is concerned with the error analysis of a numerical scheme for the approximation of parametric mean curvature flow. The scheme we study is based on a reparametrization using the DeTurck trick and was proposed by Elliott and Fritz in \cite{EF17}. 
In the semidiscrete case, for a spatial discretization by finite elements of order $k \geq 2$ we prove an optimal $H^1$--error estimate for the position vector. We present  numerical experiments that confirm this error bound and demonstrate that
the scheme has good properties with respect to the distribution of mesh points as already observed in \cite{EF17}.
\end{abstract} 


\renewcommand{\thefootnote}{\arabic{footnote}}

\section{Introduction} \label{sec:1}

Let us denote by $(\Gamma(t))_{t \in [0,T]}$ a family of closed hypersurfaces in $\mathbb R^3$. We say that
$(\Gamma(t))_{t \in [0,T]}$ evolves by mean curvature if 
\begin{eqnarray} 
V &= &  H \quad \mbox{ on } \Gamma(t), \quad 0<t \leq T, \label{eq:mcf} \\
\Gamma(0) & = & \Gamma_0. \label{eq:mcfinit}
\end{eqnarray}
Here,  $V$ and $H$ denote the normal velocity and the mean curvature of $\Gamma(t)$ respectively, while $\Gamma_0$ is a given initial hypersurface.  Equation \eqref{eq:mcf}  is one of the most important geometric
evolution laws and has been studied intensively for several decades. One of the key properties of the flow is its interpretation as the $L^2$--gradient flow for surface
area, which has led to  numerous applications in  e.g. materials science and image processing. A nice overview of theoretical aspects of the flow can be found in the monographs
by Ecker \cite{Eck04} and Mantegazza \cite{Man11}, while aspects related to the approximation of solutions by finite element methods are treated e.g. in the survey
articles \cite{DDE05} and \cite{BGN20}. \\
In what follows we shall use  a parametric representation of the evolving hypersurfaces $\Gamma(t)$, i.e. $\Gamma(t)=x(\mathcal M,t)$
for some mapping $x:\mathcal M \times [0,T] \rightarrow \mathbb R^3$, where $\mathcal M \subset \mathbb R^3$ denotes a closed reference surface. Denoting by  $H \nu$ the mean curvature vector of $\Gamma(t)$
(with $\nu$ the chosen unit normal), a common approach to study \eqref{eq:mcf} in this setting is to consider the system
\begin{equation} \label{eq:system}
x_t = H \nu  \quad \mbox{ on } \mathcal M \times (0,T].
\end{equation}
Using  the well--known formula $H \nu = \Delta_{\Gamma} \mbox{id}$, \eqref{eq:mcf}, \eqref{eq:mcfinit} then is satisfied provided that $x$ is a solution of the nonlinear parabolic initial value problem
\begin{eqnarray} 
x_t & = &  \Delta_{\Gamma(t)} \mbox{id} \circ x \quad \mbox{ on } \mathcal M \times (0,T], \label{eq:system1} \\
x(\cdot,0) & = & x_0 \qquad \qquad \; \;  \mbox{ on } \mathcal M, \label{eq:sysinit}
\end{eqnarray}
where $\Delta_{\Gamma}$ denotes the Laplace--Beltrami operator on the surface $\Gamma$ and $x_0: \mathcal M \rightarrow \mathbb R^3$ is a parametrisation of $\Gamma_0$. 
 In his seminal paper \cite{Dz90}, Dziuk uses the divergence structure
of the Laplace--Beltrami operator in order to define a  finite element approximation of mean curvature flow. The corresponding scheme employs continuous, piecewise
linear finite elements and is straightforward to implement requiring the solution of a linear system in each time step. However, a convergence analysis for this scheme
 is unfortunately still missing, a major difficulty being the fact that the system \eqref{eq:system} is only degenerate parabolic. This difficulty has
been overcome by Li and Bai \& Li in \cite{Li21, BL23}, where error estimates for Dziuk's method  were obtained, provided that  finite elements of order $k \geq 6$ are used. This result was
subsequently extended to $k \geq 3$ in \cite{BL24}. The analysis in \cite{Li21,BL23,BL24} is based on a matrix--vector form of the underlying finite element scheme, a technique 
that was previously introduced by Kov\'acs, Li and Lubich in \cite{KLL19}, where the first error analysis for parametric mean curvature flow was presented. The corresponding scheme
discretizes  an extended system that contains position,  velocity, normal and 
mean curvature as variables and uses the evolving surface finite element method  introduced in \cite{DE07} with elements of order $k \geq 2$. The analysis was recently improved
by Li and Tang in \cite{LT25} to allow elements of order $k \geq 1$ and to include an optimal $L^2$--error bound. \\
All of the approaches mentioned above have in common that the velocity vector points in normal direction, which in some cases may lead to mesh distortion at the discrete level
with a possible loss of accuracy or even breakdown of the calculation. A strategy in order to counter this effect consists in the introduction of a suitable tangential component
for the velocity vector. An important approach in this direction is the BGN--method, developed by Barrett, Garcke \& N\"urnberg in a series of papers. These schemes introduce
a tangential component at the discrete level which has good properties with respect to the distribution of mesh points, see in particular \cite{BGN08a} and \cite[Section 4.6]{BGN20}.
However, a convergence analysis in the case of surfaces seems to be open. A different approach for choosing a tangential velocity was suggested by Mikula et al. in
\cite{MRSS14}. Here, the velocity is chosen in such a way as to maintain relative volumes during the evolution which is achieved by solving an additional, appropriately defined elliptic
problem at each point in time. 
A more recent development chooses the tangential component of the velocity vector $v$ by
minimizing the energy $\int_{\Gamma(t)} | \nabla_{\Gamma(t)} v |^2$ subject to the constraint $v \cdot \nu=H$, see \cite{HL22} for a formal derivation from the BGN--approach
when the time step size tends to zero. It is shown in \cite{HL22} that the convergence analysis of \cite{KLL19} can be extended to this setting. A variant of the above idea is proposed in \cite{DL24}
and further pursued in \cite{GL25},
where instead the energy $\int_{\mathcal M} | \nabla_{\mathcal M} x(\cdot,t) |^2$ is minimised in order to obtain a conformal parametrization of $\Gamma(t)$. \\
The starting point for the analysis of this paper is an approach that was proposed by Elliott and Fritz in \cite{EF17} and which is based on the DeTurck trick. To do so, they consider the
harmonic map heat flow 
\begin{equation} \label{eq:hmhf}
\psi_t = \frac{1}{\alpha} \Delta_{\mathfrak e,\hat g(t)} \psi \quad \mbox{ on } \mathcal M \times (0,T], \qquad \psi(\cdot,0)=\mbox{id}_{\mathcal M} \quad \mbox{ on } \mathcal M,
\end{equation}
where $\alpha>0$ and $\Delta_{\mathfrak e,\hat g(t)}$ is the Laplace--Beltrami operator with respect to the metrics on $\mathcal M$ that are induced by the Euclidean metric on $\mathbb R^3$ and the metric
induced by the solution $\hat x$ of \eqref{eq:system1}, \eqref{eq:sysinit}.  One thereby obtains a family of diffeomorphisms $\psi(\cdot,t): \mathcal M \rightarrow \mathcal M$, which are used in order to reparametrize $\hat x$ via
$x(p,t):=\hat x(\psi(p,t),t)$. It is shown that $ x$ satisfies the  {\it strictly parabolic}  system
\begin{equation} \label{eq:repar}
x_t = \Delta_{\Gamma(t)} \mbox{id} \circ x + \frac{1}{\alpha} P  \Delta_{\mathcal M}  x \quad \mbox{ on } \mathcal M \times (0,T], \qquad  x(\cdot,0)=x_0 \quad \mbox{ on } \mathcal M,
\end{equation}
see \cite[Section 6]{EF17} for details.  Also, in the above $P=(I_3 - \nu \otimes \nu)$ is the projection onto the tangent space of $\Gamma(t)$. A nice feature of \eqref{eq:repar} is
the fact that the problem only involves the position vector $ x$ and doesn't require any additional variables. In \cite{EF17} finite element algorithms for \eqref{eq:repar} and
for a similar system that is derived by exchanging the roles of $\mathfrak e$ and $\hat g(t)$ are derived and compared to the BGN--method. We shall recall the
corresponding ideas in Section \ref{sec:3}. 
 Finite element algorithms for a variant of  \eqref{eq:repar} using continuous, piecewise linear elements in space and
 either an Euler or a  BDF2 method in time  have been proposed and analyzed with respect to their stability properties in \cite{Duan24} and \cite{Duan25}. Finally,  error estimates for
a finite difference scheme based on the above approach in order to approximate torus--type solutions of mean curvature flow have been obtained by Mierswa in \cite{Mie20}. \\
Our aim in this paper is to develop an error analysis for  a  finite element scheme that approximates \eqref{eq:repar}. After collecting some useful background material in Section \ref{sec:2} we shall derive in 
Section \ref{sec:3} a weak formulation of \eqref{eq:repar}, which gives rise to a natural semidiscretization using finite elements, see \eqref{eq:weak} and \eqref{eq:weakdisc}. As our
main result we shall prove in Section \ref{sec:4}  an $O(h^k)$--error bound in $H^1(\mathcal M)$ for finite elements of order $k \geq 2$. The restriction on the polynomial degree is due to the necessity of
gaining uniform control on the discrete metric tensor via an inverse estimate. Finally, we shall present a couple of numerical results in Section \ref{sec:5} which in particular confirm
the theoretical error bounds and also show that the approach has good mesh properties as already observed in \cite{EF17}.

\section{Preliminaries} \label{sec:2}
\setcounter{equation}{0}

Let $\mathcal M \subset \mathbb R^3$ be a smooth two--dimensional closed, orientable  surface with unit normal field $\mu:\mathcal M \rightarrow \mathbb R^3$. For a function
$f: \mathcal M \rightarrow \mathbb R$ we denote by $\nabla_{\mathcal M} f: \mathcal M \rightarrow \mathbb R^3$
its tangential gradient with components $\underline D_i f,i=1,2,3$. For a vectorfield $F:\mathcal M \rightarrow \mathbb R^3$
we set
\begin{displaymath}
(\nabla_{\mathcal M} F)_{ij}:= \underline D_i F_j, i,j=1,2,3  \;  \mbox{ as well as } \; \nabla_{\mathcal M} \cdot F:= \underline D_i F_i,
\end{displaymath}
while we define for a matrix--valued function $A:\mathcal M \rightarrow \mathbb R^{3 \times3}$
\begin{displaymath}
 [\nabla_{\mathcal M} \cdot A]_j := \underline D_i A_{ij}, j=1,2,3,
\end{displaymath}
where here and in what follows we  use the Einstein summation convention. \\
We shall denote the norm of the Sobolev space $W^{m,p}(\mathcal M) \, (m \in \mathbb N_0, 1 \leq p \leq \infty)$ by $\Vert \cdot \Vert_{W^{m,p}}$; in the case $p=2$ we
simply write $W^{m,2}(\mathcal M)=H^m(\mathcal M)$ as well as  $\Vert \cdot \Vert_{L^2}= \Vert \cdot \Vert$. \\
Let $x:\mathcal M \times [0,T] \rightarrow \mathbb R^3$ be a differentiable mapping such that $x(\cdot,t)$ is an  embedding for each $t \in [0,T]$  
with $\Gamma(t)=x(\mathcal M,t)$. Following \cite[(5.7)]{EF17} it will be convenient to introduce the following globally defined
metric tensor  $G=(G_{ij}): \mathcal M \times [0,T] \rightarrow \mathbb R^{3 \times 3}$ with
\begin{equation} \label{eq:defG}
G_{ij}(p,t):= \underline D_i x(p,t) \cdot \underline D_j x(p,t) + \mu_i(p) \mu_j(p), \quad i,j=1,2,3, \; (p,t)  \in \mathcal M \times [0,T],
\end{equation}
where  $a \cdot b$ denotes  the Euclidean scalar product of  $a,b  \in \mathbb R^3$. For each $(p,t) \in \mathcal M \times [0,T]$ the matrix $G(p,t)$ is invertible and we 
set $(G^{ij}):= G^{-1}$. \\
Next, let  $\varphi: U \rightarrow \mathcal M$ be a local parametrization of $\mathcal M$. We denote by
\begin{displaymath}
h_{k \ell}(\theta):= \varphi_{\theta_k}(\theta)\cdot \varphi_{\theta_{\ell}}(\theta), \quad k,\ell=1,2,
\end{displaymath}
the coefficents of the first fundamental form and set $(h^{k \ell}(\theta))=(h_{k \ell}(\theta))^{-1}$. The mapping  $\varphi$ induces the local parametrization $X(\cdot,t):=x(\cdot,t)  \circ \varphi: U \rightarrow \mathbb R^3$ 
 of $\Gamma(t)$ with first fundamental form 
 \begin{displaymath}
 g_{k \ell}(\theta,t)= X_{\theta_k}(\theta,t) \cdot X_{\theta_\ell}(\theta,t), \quad k,\ell =1,2.
 \end{displaymath}
 We can express $\underline D_i x$ in terms of the derivatives of $X$ via
\begin{equation} \label{eq:tanggrad}
\underline D_i x  \circ \varphi= h^{k \ell}  \varphi_{\theta_k}^i \, X_{\theta_\ell}, \quad 1 \leq i \leq 3.
\end{equation}
 The above formula is useful in proving the following result.

\begin{lem} With the notation introduced above we have on $U$:
\begin{eqnarray}
G_{ij} \circ \varphi & = & h^{rs} h^{pq}  g_{rp} \varphi^i_{\theta_s} \varphi^j_{\theta_q} + (\mu_i \mu_j) \circ \varphi; \quad i,j=1,2,3;  \label{eq:Gij} \\
G^{ij} \circ \varphi  & = & g^{k \ell} \varphi^i_{\theta_k} \varphi^j_{\theta_\ell} + (\mu_i  \mu_j) \circ \varphi; \quad i,j=1,2,3; \label{eq:Gijinv} \\
\mbox{det} \, G \circ \varphi & = & \frac{\mbox{det} \, (g_{ij})}{\mbox{det}\, (h_{ij})}. \label{eq:detG}
\end{eqnarray}
Furthermore, we have on $\mathcal M$:
\begin{align}
& G^{-1} \nabla_{\mathcal M} x : \nabla_{\mathcal M} x  = 2, \quad  G \mu  =  \mu;  , \label{eq:trace} \\[2mm]
& \Delta_{\Gamma} \mbox{id} \circ x  =  \frac{1}{\sqrt{\dg}} \, \nabla_{\mathcal M} \cdot  \bigl( G^{-1} \nabla_{\mathcal M} x \,  \sqrt{\dg} \bigr);   \label{eq:mcx} \\[2mm]
& \partial_t \sqrt{\dg} = G^{-1} \nabla_{\mathcal M} x : \nabla_{\mathcal M} x_t \sqrt{\dg}. \label{eq:dtdet}
\end{align}
Here $A : B$ denotes the Frobenius inner product of $A,B \in \mathbb R^{3 \times 3}$. 
\end{lem}
\begin{proof} The relation \eqref{eq:Gij} follows immediately from \eqref{eq:tanggrad}. Next, \eqref{eq:Gij} together with the relation $\varphi^j_{\theta_\ell} \mu_j \circ \varphi=0$ implies
\begin{eqnarray*}
\lefteqn{  \hspace{-3cm} 
G_{ij} \circ \varphi \bigl( g^{k \ell} \varphi^j_{\theta_k} \varphi^m_{\theta_\ell} + (\mu_j  \mu_m) \circ \varphi \bigr) = \bigl( h^{rs} h^{pq}  g_{rp} \varphi^i_{\theta_s} \varphi^j_{\theta_q} + (\mu_i  \mu_j) \circ \varphi \bigr)
 \bigl( g^{k \ell} \varphi^j_{\theta_k} \varphi^m_{\theta_\ell} + (\mu_j  \mu_m) \circ \varphi \bigr) }  \\
& = & h^{rs} \varphi^i_{\theta_s} \varphi^m_{\theta_r} + (\mu_i  \mu_m) \circ \varphi= :a_{im}, \; i,m=1,2,3.
\end{eqnarray*}
It is easily seen that $a_{im} \varphi^m_{\theta_k}=\varphi^i_{\theta_k}, k=1,2$ and $a_{im} \mu_m=\mu_i$ so that $a_{im}=\delta_{im}$ and hence \eqref{eq:Gijinv} follows, while
 \eqref{eq:detG} is shown in \cite[(5.11)]{EF17}. Furthermore, using \eqref{eq:tanggrad} and \eqref{eq:Gijinv} we have
\begin{equation} \label{eq:Gijx}
(G^{ij} \underline D_j x) \circ \varphi = g^{k \ell} X_{\theta_{\ell}} \varphi^i_{\theta_k}, \quad i=1,2,3
\end{equation}
and therefore
\begin{displaymath}
\bigl( G^{-1} \nabla_{\mathcal M} x : \nabla_{\mathcal M} x \bigr) \circ \varphi  = \bigl( G^{ij} \underline D_j x \cdot \underline D_i x \bigr)  \circ \varphi = g^{k \ell} \varphi^i_{\theta_k} h^{rs} \varphi^i_{\theta_s}    X_{\theta_{\ell}} \cdot X_{\theta_r} =
g^{k \ell} g_{\ell r} h^{rs} h_{ks} = h^{ks} h_{ks}=2.
\end{displaymath}
Next,  the relation $G \mu = \mu$ is a direct consequence of \eqref{eq:Gij}, while a long but straightforward calculation starting from \eqref{eq:Gijx}
yields
\begin{displaymath}
\frac{1}{\sqrt{\mbox{det}G}} \, \underline D_i \bigl( G^{ij} \, \underline D_j x \,  \sqrt{\mbox{det} G} \bigr) \circ \varphi = \frac{1}{\sqrt{\mbox{det} g}} \bigl( g^{k \ell} ( x \circ \varphi)_{\theta_\ell} \sqrt{\mbox{det} g} \bigr)_{\theta_k} = \bigl( \Delta_\Gamma \mbox{id} \circ x \bigr) \circ \varphi,
\end{displaymath}
so that \eqref{eq:mcx} follows. Finally, a well--known formula for the derivative of the determinant yields
\begin{displaymath}
\partial_t \sqrt{\dg} = \frac{1}{2} \mbox{trace} \bigl( G^{-1} \partial_t G \bigr) \sqrt{\dg} =  G^{ij} \underline D_j x \cdot \underline D_i x_t \sqrt{\dg} = G^{-1} \nabla_{\mathcal M} x : \nabla_{\mathcal M} x_t \sqrt{\dg},
\end{displaymath}
where we also used the symmetry of $G^{-1}$. 
 \end{proof}
 
 \vspace{2mm}

\begin{lem} \label{lem:Gprop} Let $v_j:= \underline D_{j+1} x \wedge \underline D_{j+2} x, j=1,2,3$, where indices are taken modulo 3. Then
$\nu: \mathcal M \times [0,T] \rightarrow \mathbb R^3$ with
\begin{equation} \label{eq:defnu}
\nu := \frac{1}{\sqrt{\dg}} \mu_j v_j
\end{equation}
defines a unit normal field to $\Gamma(t)$. Furthermore, we have 
\begin{equation} \label{eq:form1} 
G^{ik} G^{\ell j} \underline D_j x \wedge \underline D_k x  =  (-1)^{i+\ell} \frac{1}{\sqrt{\mbox{det} \, G}} \mu_{\sigma(i,\ell)} \nu, \; 1 \leq i < \ell \leq 3,
\end{equation}
 where $\sigma(1,2)=3, \sigma(1,3)=2, \sigma(2,3)=1$. 
\end{lem}
\begin{proof} The proof is given in the appendix.
\end{proof}

\section{Weak formulation and finite element approximation} \label{sec:3}
\setcounter{equation}{0}

Let us return to the evolution law \eqref{eq:mcf} and suppose that \eqref{eq:repar}
has a smooth solution $x: \mathcal M \times [0,T] \rightarrow \mathbb R^3$.  Recalling that $\Delta_{\Gamma} \mbox{id} \circ x  = H \nu$ we infer from \eqref{eq:repar} that 
\begin{displaymath}
x_t \cdot \nu = \Delta_{\Gamma} \mbox{id} \circ  x \cdot \nu, \quad P x_t = \frac{1}{\alpha} P \Delta_{\mathcal M} x, \quad \mbox{ where } P=I_3 - \nu \otimes \nu.
\end{displaymath}
Splitting the time derivative in a similar way as in \cite[Section 2]{EF17} and using \eqref{eq:mcx} it is easily checked that   \eqref{eq:repar} can be equivalently written as
\begin{eqnarray}
\lefteqn{
\frac{\alpha+1}{\sqrt{\dg}} x_t + (1- \frac{\alpha+1}{\sqrt{\dg}}) (x_t \cdot \nu) \nu } \label{eq:repar1} \\
& = &  \frac{1}{\sqrt{\dg}} \, \nabla_{\mathcal M} \cdot  \bigl( G^{-1} \nabla_{\mathcal M} x \,  \sqrt{\dg} \bigr) +   \frac{1}{\sqrt{\dg}}
\frac{\alpha+1}{\alpha}  P \Delta_{\mathcal M} x. \nonumber
\end{eqnarray}
Here, $G, \nu$ are given by \eqref{eq:Gij} and \eqref{eq:defnu} respectively.
Note that \eqref{eq:repar1} slightly differs from \cite[(6.4)]{EF17} in that the second term on the right hand side is multiplied by $\frac{\alpha+1}{\alpha}$. The fact that this factor is larger than 1 will be exploited in
the error analysis. \\
\noindent
Multiplying  \eqref{eq:repar1} by $\chi  \sqrt{\dg}$, where $\chi \in H^1(\mathcal M; \mathbb R^3) $, and integrating by parts
on $\mathcal M$ we derive the following weak formulation:
\begin{eqnarray}
\lefteqn{ \hspace{-2cm} 
\int_{\mathcal M} A(\nu, \sqrt{\dg}) x_t \cdot \chi  \, do + \int_{\mathcal M} G^{-1} \nabla_{\mathcal M} x :  \nabla_{\mathcal M} \chi \sqrt{\dg} \, do } \nonumber \\
& &  +  \frac{\alpha+1}{\alpha}  \int_{\mathcal M} 
\nabla_{\mathcal M} x : \nabla_{\mathcal M} \bigl( P \chi \bigr) \, do =0 \qquad \forall \chi \in H^1(\mathcal M;\mathbb R^3). \label{eq:weak}
\end{eqnarray}
In the above, $A(w,\rho) \in \mathbb R^{3 \times 3}$ is given by
\begin{equation} \label{eq:defH}
A(w,\rho):= (\alpha+1) I_3 +(\rho - \alpha -1) w \otimes w, \quad w \in \mathbb R^3, \rho>0.
\end{equation}
Note that $A(w,\rho)$ is symmetric and positive definite with
\begin{equation} \label{eq:Hpos}
A(w,\rho) \xi \cdot \xi = (\alpha+1) (| \xi |^2 - (\xi \cdot w)^2) + \rho (\xi \cdot w)^2 \geq \min (1, \rho) | \xi |^2 \quad \mbox{ for all } \xi \in \mathbb R^3.  
\end{equation}

\noindent
In order to discretize in space we fix  $k \in \mathbb N$ and denote by  
$(S^k_h)_{0<h \leq h_0} \subset W^{1,\infty}(\mathcal M)$  a sequence of finite--dimensional spaces, for which we make the following assumptions: \\[2mm]
{\bf (A1)}: There exists an interpolation operator $I^k_h: C^0(\mathcal M) \rightarrow S^k_h$ such that
for $2 \leq r \leq \infty$:
\begin{equation}  \label{eq:interpol} 
\Vert f - I^k_h f \Vert_{L^r}  + h\Vert \nabla_{\mathcal M} (f - I^k_h f) \Vert_{L^r}  \leq  \hat c h^{k+1} \Vert f \Vert_{W^{k+1,r}}  \quad   \forall \, f \in W^{k+1,r}(\mathcal M).
\end{equation}
{\bf (A2)} For $2 \leq r \leq \infty$ there holds:
\begin{equation}  \label{eq:inverse}
\Vert \chi_h \Vert_{L^r} \leq \hat c h^{\frac{2}{r}-1} \Vert \chi_h \Vert, \; \Vert \nabla_{\mathcal M}  \chi_h \Vert_{L^r}   \leq     \hat c h^{-1} \Vert \chi_h \Vert_{L^r} \quad \quad  \forall \chi_h \in S^k_h.
\end{equation} 
Possible choices of $S^k_h$ and $I^k_h$  will be given in Section \ref{sec:5}. Note that for a mapping $x_h \in (S^k_h)^3$ with $\dgh > 0$ the associate normal 
\begin{equation}
\displaystyle \nu_h=\frac{1}{\sqrt{\dgh}} \,  \mu_j \underline D_{j+1} x_h \wedge \underline D_{j+2} x_h
\label{nueq}
\end{equation}
will in general only belong to $L^\infty(\mathcal M)^3$ so that we cannot
expect the components of  $I_3 - \nu_h \otimes \nu_h$ to have a weak derivative which would be required for the third term on the left hand side of \eqref{eq:weak}. In order to
deal with this issue we make the following final assumption: \\[2mm]
{\bf (A3)} Given $x_h \in (S^k_h)^3, x \in W^{k+1,\infty}(\mathcal M)^3$ with $\dg, \dgh \geq c_0, | \nabla_{\mathcal M} x_h | \leq c_1$ there exists $\hat \nu_h \in (S^k_h)^3$ such that
\begin{equation}  \label{eq:tnuh}
\Vert \hat \nu_h - \nu \Vert \leq c \bigl( \Vert \nabla_{\mathcal M} (x_h - x) \Vert + h^k \bigr), \qquad \nu=\frac{1}{\sqrt{\dg}} \,  \mu_j \underline D_{j+1} x \wedge \underline D_{j+2} x,
\end{equation}
where the constant $c$ only depends on $c_0,c_1, \Vert x \Vert_{W^{k+1,\infty}}$ and the constant $\hat c$ from {\bf (A1), (A2)}. \\[2mm]
This assumption is tailored to the error analysis and will be verified for the case $k=2$ in Section \ref{sec:5}. 
Note that $P_h:=I_3 - \hat \nu_h \otimes \hat \nu_h \in W^{1,\infty}(\mathcal M,\mathbb R^{3 \times 3})$. \\[2mm]
Our semidiscrete problem reads: find $x_h: \mathcal M \times [0,T] \rightarrow \mathbb R^3$ such that  $x_h(\cdot,t) \in (S^k_h)^3$ for all $t \in [0,T]$,  $x_h(\cdot,0)=I^k_h x_0$ and
\begin{eqnarray} 
\lefteqn{ \hspace{-2cm}
\int_{\mathcal M} A(\nu_h,\sqrt{\dgh}) x_{h,t} \cdot \chi_h \, do + \int_{\mathcal M} G_h^{-1} \nabla_{\mathcal M}  x_h : \nabla_{\mathcal M} \chi_h \sqrt{\dgh} \, do }
\nonumber  \\
& &  + \frac{\alpha+1}{\alpha}  \int_{\mathcal M}  \nabla_{\mathcal M} x_h : \nabla_{\mathcal M}  \bigl( P_h \chi_h \bigr) \, do =0  \qquad \forall \chi_h \in (S^k_h)^3. \label{eq:weakdisc}
\end{eqnarray}
Here, $G_h$ and $P_h$ are given by 
\begin{displaymath}
G_{h,ij}  =   \underline D_i x_h \cdot \underline D_j x_h + \mu_i \mu_j, \quad i,j=1,2,3; \quad 
P_h  =  I_3 - \hat \nu_h \otimes \hat \nu_h
\end{displaymath}
with $\hat \nu_h$ as in {\bf (A3)}.
Our main result reads as follows:
\begin{thm} \label{thm:main}
Let $k \geq 2$ and suppose that \eqref{eq:repar} has a smooth solution $x: \mathcal M \times [0,T] \rightarrow \mathbb R^3$.
 Then there exists $h_0>0$ such that for all $0<h \leq h_0$ the semi--discrete
problem \eqref{eq:weakdisc} has a unique solution $x_h: \mathcal M \times [0,T] \rightarrow \mathbb R^3$ and the following error estimates hold:
\begin{displaymath}
\int_0^T \Vert x_{h,t} - x_t \Vert^2 dt + \max_{0 \leq t \leq T} \Vert (x_h-x)(t) \Vert_{H^1}^2 \leq c h^{2k}.
\end{displaymath}
\end{thm}


\section{Error analysis} \label{sec:4}
\setcounter{equation}{0}

Since we have assumed that $x: \mathcal M  \times [0,T] \rightarrow \mathbb R^3$ is a smooth embedding there exist $c_0, c_1>0$ such that
\begin{align}
& \sqrt{\dg} \geq c_0, \quad G^{-1} \xi \cdot \xi \geq c_0 | \xi |^2 \quad \forall \xi \in \mathbb R^3 \quad \mbox{ on } \mathcal M \times [0,T]; \label{eq:prop1} \\
& | \nabla_{ \mathcal M} x | \leq c_1 \quad \mbox{ on } \mathcal M \times [0,T]; \quad \int_0^T \Vert x_t \Vert_{L^\infty}^2 dt \leq c_1. \label{eq:prop2}
\end{align}
In what follows we use the abbreviations
\begin{displaymath}
 e:=x_h-x \mbox{ and } e_h:=x_h - I^k_h x
 \end{displaymath}
 and set
\begin{displaymath}
\hat T_h = \sup \Bigl\{  t \in [0,T]   \,  | \, x_h \mbox{ exists on } [0,t],  \Vert \nabla_{\mathcal M} e(s) \Vert  \leq h^{\frac{3}{2}} \mbox{ for } 0 \leq s \leq t,
 \int_0^t \Vert x_{h,t} \Vert_{L^\infty}^2 ds \leq 2c_1  \Bigr\}. 
\end{displaymath}
Clearly, $\hat T_h>0$ and the definition of $\hat T_h$ implies that $x_h$ can be extended to the interval $[0,\hat T_h]$. 
Using \eqref{eq:interpol} and \eqref{eq:inverse} we obtain on $[0,\hat T_h]$
\begin{eqnarray}
\Vert \nabla_{\mathcal M} e \Vert_{L^\infty} & \leq &   \Vert \nabla_{\mathcal M} (x_h - I^k_h x) \Vert_{L^\infty} + \Vert \nabla_{\mathcal M} (x - I^k_h x) \Vert_{L^\infty}  
 \leq  \hat c h^{-1} \Vert \nabla_{\mathcal M} (x_h - I^k_h x) \Vert +  c h^k  \nonumber \\ 
& \leq & \hat c h^{-1} \bigl( \Vert \nabla_{\mathcal M} e \Vert + \Vert \nabla_{\mathcal M} (x - I^k_h x)  \Vert \bigr) +  c h^k
 \leq \hat c h^{\frac{1}{2}}+ c h^{k-1} \leq c h^{\frac{1}{2}},  \label{eq:este1}
\end{eqnarray}
since $ k \geq 2$. Note that  $c$ only depends on $x$ and $\hat c$. In particular we infer from \eqref{eq:prop1} and \eqref{eq:prop2} that there exists $h_0>0$ such that for $0<h \leq h_0$
\begin{equation} \label{eq:estxh1}
\max_{0 \leq t \leq \hat T_h} \Vert \nabla_{\mathcal M} x_h(\cdot,t)  \Vert_{L^\infty} \leq 2c_1, \quad \sqrt{\dgh} \geq \frac{c_0}{2} \mbox{ on } \mathcal M \times [0,\hat T_h].
\end{equation}

\noindent
Taking the difference of \eqref{eq:weakdisc} and \eqref{eq:weak} 
we obtain the following error equation:
\begin{eqnarray}
\lefteqn{ \hspace{-2cm} 
\int_{\mathcal M} A(\nu,\sqrt{\dg}) e_t \cdot \chi_h  \, do + \int_{\mathcal M} \bigl( G_h^{-1} \nabla_{\mathcal M}  x_h \sqrt{\dgh} - G^{-1} \nabla_{\mathcal M}  x \sqrt{\dg} \bigr) : 
\nabla_{\mathcal M}  \chi_h  \, do } \nonumber \\
&  & +  \frac{\alpha+1}{\alpha}  \int_{\mathcal M}  \bigl( \nabla_{\mathcal M}  x_h : \nabla_{\mathcal M}  \bigl( P_h \chi_h \bigr) - \nabla_{\mathcal M}  x : \nabla_{\mathcal M}  ( P \chi_h ) \bigr)  \, do \nonumber \\
& = &   \int_{\mathcal M} \bigl( A(\nu,\sqrt{\dg}) - A(\nu_h,\sqrt{\dgh}) \bigr) x_{h,t} \cdot \chi_h \, do \qquad \forall \chi_h \in (S^k_h)^3. \label{eq:erroreq}  
\end{eqnarray}
In a first step we derive a bound for $\int_0^{\hat T_h} \Vert \nabla_{\mathcal M} e \Vert^2 dt$, which will be used for  a Gronwall argument in the second part of the analysis.

\begin{lem} There exist $h_0>0$ and $c_2$ which only depend on the solution $x$, $T$ and the constant $\hat c$  from {\bf (A1)}, {\bf (A2)} such that for $0 < h \leq h_0$ we have 
\begin{equation} \label{eq:esth1}
\sup_{0 \leq t \leq  \hat T_h} \Vert e_h(t) \Vert^2 + \int_0^{\hat T_h} \Vert \nabla_{\mathcal M} e(t) \Vert^2 dt + h^2 \int_0^{\hat T_h} \Vert \nabla_{\mathcal M} e \Vert_{L^\infty}^2 dt  \leq c_2 h^{2k}.
\end{equation}
\end{lem}
\begin{proof}
Let us begin by inserting $\chi_h = e_h$ into \eqref{eq:erroreq}:
\begin{eqnarray}
\lefteqn{ 
\int_{\mathcal M} A(\nu,\sqrt{\dg}) e_{h,t} \cdot e_h  \, do  + \int_{\mathcal M} \bigl( G_h^{-1}\nabla_{\mathcal M}  x_h \sqrt{\dgh} - G^{-1} \nabla_{\mathcal M}  x \sqrt{\dg} \bigr) : 
\nabla_{\mathcal M} e   \, do }  \label{eq:err1}  \\
& & \frac{\alpha+1}{\alpha}  \int_{\mathcal M}  \bigl( \nabla_{\mathcal M}  x_h : \nabla_{\mathcal M} (P_h e_h) - \nabla_{\mathcal M}   x :  \nabla_{\mathcal M} ( P e_h) \bigr)  \, do \nonumber  \\
& = &  \int_{\mathcal M} A(\nu,\sqrt{\dg}) (x_t - I^k_h x_t)  \cdot e_h \, do +  \int_{\mathcal M} \bigl( A(\nu,\sqrt{\dg}) - A(\nu_h,\sqrt{\dgh}) \bigr) x_{h,t} \cdot e_h  \, do \nonumber \\
& &+  \int_{\mathcal M} \bigl( G_h^{-1} \nabla_{\mathcal M}  x_h \sqrt{\dgh} - G^{-1}  \nabla_{\mathcal M}  x \sqrt{\dg} \bigr) :  \nabla_{\mathcal M}  (I^k_h x - x)  \, do  \equiv :T_1+T_2+T_3. \nonumber
\end{eqnarray}
For the first integral on the left hand side of \eqref{eq:err1} we  derive using the symmetry of $A$ 
\begin{eqnarray}
\lefteqn{ \hspace{-4cm} 
\int_{\mathcal M} A(\nu,\sqrt{\dg}) e_{h,t} \cdot e_h  \, do = \frac{1}{2} \frac{d}{dt} \int_{\mathcal M} A(\nu,\sqrt{\dg})  e_h \cdot e_h  \, do - \frac{1}{2} \int_{\mathcal M} \partial_t A(\nu,\sqrt{\dg})  e_h \cdot e_h  \, do } \nonumber \\
& \geq &  \frac{1}{2} \frac{d}{dt} \int_{\mathcal M} A(\nu,\sqrt{\dg}) e_h \cdot e_h  \, do 
 - c \Vert e_h \Vert^2. \label{eq:e1}
\end{eqnarray}
Next, in order to handle the second integral we use a Taylor expansion argument together with \eqref{eq:este1} and \eqref{eq:estxh1}. Choosing $h_0$ smaller if necessary we derive
\begin{eqnarray}
G_h^{-1} & = & G^{-1} - G^{-1}(G_h-G) G^{-1} + R_h, \qquad  \qquad \quad  | R_h | \leq c | \nabla_{\mathcal M} e |^2, \label{eq:ghinvexp} \\[2mm]
\sqrt{\dgh} & = & \sqrt{\dg} + G^{-1} \nabla_{\mathcal M} x: \nabla_{\mathcal M} e  \sqrt{\dg}+ r_h, \quad | r_h | \leq c | \nabla_{\mathcal M} e |^2, \label{eq:detexp}
\end{eqnarray} 
where $c$ is independent of $h$. Thus we obtain 
\begin{eqnarray*}
\lefteqn{
G_h^{-1}\nabla_{\mathcal M}  x_h \sqrt{\dgh} - G^{-1} \nabla_{\mathcal M}  x \sqrt{\dg} = G^{-1} \nabla_{\mathcal M} e \sqrt{\dg} } \\
  &  &    -G^{-1} (G_h-G) G^{-1} \nabla_{\mathcal M} x \sqrt{\dg} + G^{-1} \nabla_{\mathcal M} x ( G^{-1} \nabla_{\mathcal M} x : \nabla_{\mathcal M} e) \sqrt{\dg} + R^{(1)}_h
 \end{eqnarray*}
 where $| R^{(1)}_h | \leq c | \nabla_{\mathcal M} e |^2$. Taking the Frobenius inner product with $\nabla_{\mathcal M}e $ we derive
 \begin{eqnarray}
 \lefteqn{
 \bigl( G_h^{-1}\nabla_{\mathcal M}  x_h \sqrt{\dgh} - G^{-1} \nabla_{\mathcal M}  x \sqrt{\dg} \bigr) : \nabla_{\mathcal M} e } \label{eq:err2} \\
 & = & G^{ik} \underline D_k e \cdot \underline D_i e \sqrt{\dg} - G^{ik}  \bigl( \underline D_k x_h \cdot \underline D_\ell x_h - \underline D_k x \cdot \underline D_\ell x \bigr) G^{j \ell} (\underline D_j x 
 \cdot \underline D_i e) \sqrt{\dg} \nonumber  \\
 & & + G^{ik} G^{j \ell} (\underline D_j x \cdot \underline D_\ell e) (\underline D_k x \cdot \underline D_i e) \, \sqrt{\dg} + r^{(1)}_h, \nonumber \\
 & = & G^{ik} \underline D_k e \cdot \underline D_i e \sqrt{\dg} - G^{ik}  \bigl( \underline D_k e \cdot \underline D_\ell x + \underline D_k x \cdot \underline D_\ell e \bigr) G^{j \ell} (\underline D_j x 
 \cdot \underline D_i e) \sqrt{\dg} \nonumber  \\
 & & + G^{ik} G^{j \ell} (\underline D_j x \cdot \underline D_\ell e) (\underline D_k x \cdot \underline D_i e) \, \sqrt{\dg} + r^{(2)}_h, \nonumber
 \end{eqnarray}
 where $| r^{(1)}_h |, | r^{(2)}_h | \leq c | \nabla_{\mathcal M} e |^3$.
Let us decompose
\begin{displaymath}
\underline D_i e = G^{j \ell} ( \underline D_i e \cdot \underline D_j x) \underline D_\ell x + (\underline D_i e \cdot \nu) \nu, \quad i=1,2,3
\end{displaymath}
from which we deduce observing that $\underline D_i x \cdot \nu=0$ and $G^{k \ell} \mu_\ell \underline D_k x= \mu_k \underline D_k x =0$ (recall 
\eqref{eq:tanggrad})  that
\begin{eqnarray} 
G^{ik} \underline D_i e \cdot \underline D_k e & = &  G^{ik} G^{j \ell} G^{rs}  (\underline D_i e \cdot \underline D_j x)(\underline D_k e \cdot \underline D_r x)  \, (\underline D_\ell x \cdot \underline D_s x) + 
G^{ik} (\underline D_i e \cdot \nu) (\underline D_k e \cdot \nu) \nonumber \\
& = & G^{ik} G^{j \ell} G^{rs} (G_{\ell s} - \mu_\ell \mu_s) (\underline D_i e \cdot \underline D_j x)(\underline D_k e \cdot \underline D_r x) + G^{ik} (\underline D_i e \cdot \nu) (\underline D_k e \cdot \nu) \nonumber  \\
& = & G^{ik} G^{j \ell}  (\underline D_i e \cdot \underline D_j x)(\underline D_k e \cdot \underline D_\ell x) + G^{ik} (\underline D_i e \cdot \nu) (\underline D_k e \cdot \nu). \label{eq:decompe}
\end{eqnarray} 
Inserting this relation into \eqref{eq:err2}  and simplifying we obtain using the Binet--Cauchy identity
\begin{eqnarray}
\lefteqn{ \hspace{-1cm} 
 \bigl( G_h^{-1}\nabla_{\mathcal M}  x_h \sqrt{\dgh} - G^{-1} \nabla_{\mathcal M}  x \sqrt{\dg} \bigr) : \nabla_{\mathcal M} e 
  =   G^{ik}  (\underline D_i e \cdot \nu) (\underline D_k e \cdot \nu)  \sqrt{\dg} } \label{eq:err3}  \\
  & & +      G^{ik} G^{j \ell}\bigl( (\underline D_j x \cdot \underline D_\ell e) (\underline D_k x \cdot \underline D_i e) - (\underline D_k x \cdot \underline D_\ell e) (\underline D_j x \cdot \underline D_i e) \bigr)  \sqrt{\dg} 
  + r^{(2)}_h \nonumber  \\
  & = &  G^{ik} (\underline D_i e \cdot \nu) (\underline D_k e \cdot \nu)  \sqrt{\dg} + G^{ik} G^{j \ell} (\underline D_j x \wedge \underline D_k x) \cdot ( \underline D_\ell e \wedge \underline D_i e) \sqrt{\dg}
  + r^{(2)}_h. \nonumber
 \end{eqnarray}
 In order to handle the second term on the right hand side of \eqref{eq:err3} we claim that 
 \begin{equation} \label{eq:estpe}
 | G^{ik} G^{j \ell} (\underline D_j x \wedge \underline D_k x) \cdot ( \underline D_\ell e \wedge \underline D_i e)  | \sqrt{\dg}  \leq \sum_{i=1}^3 | P \underline D_i e |^2.
 \end{equation}
 In order to prove this inequality we  choose $v,w \in \mathbb R^3$ in such a way that $\lbrace v,w,\nu \rbrace$ is an orthonormal basis of $\mathbb R^3$ with $\nu=v \wedge w$ and
 let $a,b \in \mathbb R^3$ be defined by $a_i:=P \underline D_i e \cdot v, b_i=P \underline D_i  e\cdot w, i=1,2,3$. Thus $P \underline D_i e = a_i v + b_i w$ and 
 we have for $1 \leq i < \ell \leq 3$ with $\sigma(i,\ell)$ as in Lemma \ref{lem:Gprop}
 \begin{displaymath}
 \nu \cdot   ( \underline D_\ell e \wedge \underline D_i e) =\nu \cdot   (P \underline D_\ell e \wedge P\underline D_i e) = (a_\ell b_i - a_i b_\ell) = (-1)^{i+\ell} (a \wedge b)_{\sigma(i,\ell)}.
 \end{displaymath}
 Combining this relation with \eqref{eq:form1} we obtain
 \begin{eqnarray*}
 \lefteqn{ 
 G^{ik} G^{j \ell} (\underline D_j x \wedge \underline D_k x) \cdot ( \underline D_\ell e \wedge \underline D_i e)
 = 2  \sum_{1 \leq i < \ell \leq 3}  G^{ik} G^{j \ell} (\underline D_j x \wedge \underline D_k x) \cdot ( \underline D_\ell e \wedge \underline D_i e)}  \\
 &=& \frac{2}{\sqrt{\mbox{det} \, G}}  \sum_{1 \leq i < \ell \leq 3} (-1)^{i+\ell} \mu_{\sigma(i,\ell)} \, \nu \cdot ( \underline D_\ell e \wedge \underline D_i e) 
 =  \frac{2}{\sqrt{\mbox{det} \, G}} \sum_{1 \leq i < \ell \leq 3} \mu_{\sigma(i,\ell)} (a \wedge b)_{\sigma(i,\ell)} .
 \end{eqnarray*}
Finally, the Cauchy Schwarz inequality together with the fact that $| \mu| =1$ implies
 \begin{displaymath}
 |  G^{ik} G^{j \ell} (\underline D_j x \wedge \underline D_k x) \cdot ( \underline D_\ell e \wedge \underline D_i e) |  \, \sqrt{\mbox{det} \, G}
    \leq 2 | a \wedge b | 
  \leq   \bigl( | a |^2 + | b |^2 \bigr)  =  \sum_{i=1}^3 | P \underline D_i e |^2,
 \end{displaymath}
 which completes the proof of \eqref{eq:estpe}. \\
\noindent
 If we now integrate \eqref{eq:err3} over $\mathcal M$ and use \eqref{eq:estpe}, \eqref{eq:este1} as well as \eqref{eq:prop1} we obtain
\begin{eqnarray}
\lefteqn{ \hspace{-1cm} 
\int_{\mathcal M} \bigl( G_h^{-1}\nabla_{\mathcal M}  x_h \sqrt{\dgh} - G^{-1} \nabla_{\mathcal M}  x \sqrt{\dg} \bigr) : 
\nabla_{\mathcal M} e   \, do } \label{eq:err4} \\
&  \geq&  \int_{\mathcal M} G^{ik} (\underline D_i e \cdot \nu) (\underline D_k e \cdot \nu)  \sqrt{\dg} 
 -   \sum_{i=1}^3 \int_{\mathcal M} | P \underline D_i e |^2 \, do - c \int_{\mathcal M} | \nabla_{\mathcal M} e |^3 \, do \nonumber  \\
& \geq & c_0^2 \sum_{i=1}^3 \int_{\mathcal M} (\underline D_i e \cdot \nu)^2 \, do  -  \sum_{i=1}^3 \int_{\mathcal M} | P \underline D_i e |^2 \, do - c h^{\frac{1}{2}}   \Vert \nabla_{\mathcal M} e \Vert^2. \nonumber
\end{eqnarray}

\noindent
Let us next turn to the third integral on the left hand side of \eqref{eq:err1} and write
\begin{eqnarray}
\lefteqn{ \hspace{-1.5cm} 
\int_{\mathcal M}  \bigl( \nabla_{\mathcal M}  x_h : \nabla_{\mathcal M}  \bigl( P_h e_h \bigr) - \nabla_{\mathcal M}  x :  \nabla_{\mathcal M} \bigl( P e_h \bigr) \bigr)  \, do  = \int_{\mathcal M} \nabla_{\mathcal M}  e :\nabla_{\mathcal M} \bigl( (P_h-P) e_h \bigr) \, do } \nonumber  \\
& &  + \int_{\mathcal M} \nabla_{\mathcal M}  e : \nabla_{\mathcal M} \bigl( P e_h \bigr) \, do  + \int_{\mathcal M} \nabla_{\mathcal M}  x :  \nabla_{\mathcal M}  \bigl( (P_h - P) e_h \bigr) \, do
\nonumber   \\
& \equiv : & S_1+ S_2 + S_3. \label{eq:s1s2s3}
\end{eqnarray}
We recall that $P_h = I_3 - \hat \nu_h \otimes \hat \nu_h$ so that $(P_h - P) e_h = (e_h \cdot \nu) \nu - (e_h \cdot \hat \nu_h) \hat \nu_h$ and hence
\begin{eqnarray}
\nabla_{\mathcal M} \bigl( (P_h - P) e_h \bigr) & = & \nabla_{\mathcal M} [ e_h \cdot \nu] \otimes (\nu - \hat \nu_h) + \nabla_{\mathcal M} [ e_h \cdot (\nu - \hat \nu_h)] \otimes  \hat \nu_h \nonumber  \\
& & + (e_h \cdot \hat \nu_h) \nabla_{\mathcal M}(\nu - \hat \nu_h) + \bigl(  e_h \cdot (\nu - \hat \nu_h)  \bigr) \nabla_{\mathcal M} \nu. \label{eq:phpdif}
\end{eqnarray} 
Before estimating $S_1$ we observe that we have  for  $2 \leq r \leq \infty$:
\begin{equation} \label{eq:tnuh1}
\Vert \hat \nu_h - \nu \Vert_{L^r} + h \Vert \nabla_{\mathcal M} ( \hat \nu_h - \nu) \Vert_{L^r} \leq c h^{\frac{2}{r}-1} \Vert \nabla_{\mathcal M} e \Vert + c h^{k-1+\frac{2}{r}}.
\end{equation}
To see this, we use \eqref{eq:inverse}, \eqref{eq:interpol},  \eqref{eq:tnuh}  in order to estimate
\begin{eqnarray*}
\Vert \hat  \nu_h - I^{k-1}_h \nu \Vert_{L^r} & \leq &  c h^{\frac{2}{r}-1} \Vert \hat \nu_h -  I^{k-1}_h \nu  \Vert    
 \leq  c h^{\frac{2}{r}-1} \Vert \hat \nu_h - \nu \Vert + c h^{\frac{2}{r}-1} h^k  \\
 & \leq & c h^{\frac{2}{r}-1} \Vert \nabla_{\mathcal M} (x_h-x) \Vert + c h^{k-1+\frac{2}{r}},
\end{eqnarray*}
and similarly
\begin{displaymath}
\Vert \nabla_{\mathcal M} ( \hat \nu_h - I^{k-1}_h \nu) \Vert_{L^r} \leq c h^{\frac{2}{r}-2}  \Vert \nabla_{\mathcal M} (x_h-x) \Vert + c h^{k-2+\frac{2}{r}}.
\end{displaymath}
The bound \eqref{eq:tnuh1} then follows from \eqref{eq:interpol}. Note that we have in particular that  $\Vert \hat \nu_h \Vert_{L^\infty} \leq c$ so that
we infer from \eqref{eq:phpdif}
\begin{equation} \label{eq:phmp}
| \nabla_{\mathcal M} \bigl( (P_h - P) e_h \bigr) | \leq c \bigl( | e_h | \, | \nu - \hat \nu_h| + | \nabla_{\mathcal M} e_h | \, | \nu - \hat \nu_h | + |e_h | | \nabla_{\mathcal M} (\nu - \hat \nu_h) | \bigr).
\end{equation}
Using this estimate together with \eqref{eq:tnuh1},  \eqref{eq:interpol}, the bound $\Vert \nabla_{\mathcal M} e \Vert \leq h^{\frac{3}{2}}$ and recalling that $k \geq 2$ we obtain
\begin{eqnarray}
| S_1 | & \leq & c \int_{\mathcal M} | \nabla_{\mathcal M} e| \bigl(  | e_h | \, | \nu - \hat \nu_h| + | \nabla_{\mathcal M} e_h | \, | \nu - \hat \nu_h | + |e_h | | \nabla_{\mathcal M} (\nu - \hat \nu_h) |   \bigr) do  \nonumber  \\
& \leq & c \Vert \nabla_{\mathcal M} e \Vert \bigl( \Vert e_h \Vert_{H^1} \Vert \nu - \hat \nu_h \Vert_{L^\infty} + \Vert \nabla_{\mathcal M} (\hat \nu_h - \nu ) \Vert_{L^3} \Vert e_h \Vert_{L^6}   \bigr) \nonumber  \\
& \leq & c \Vert \nabla_{\mathcal M} e \Vert \Vert e_h \Vert_{H^1}  \bigl( \Vert \hat \nu_h - \nu \Vert_{L^\infty}  +  \Vert \nabla_{\mathcal M} (\hat \nu_h - \nu ) \Vert_{L^3}  \bigr) \nonumber  \\
& \leq & c \Vert \nabla_{\mathcal M} e \Vert \Vert e_h \Vert_{H^1} \bigl( h^{-\frac{4}{3}} \Vert \nabla_{\mathcal M} e \Vert + h^{\frac{2}{3}}+ h^{-1}  \Vert \nabla_{\mathcal M} e \Vert+h \bigr) \nonumber \\
& \leq & c h^{\frac{1}{6}} \Vert \nabla_{\mathcal M} e \Vert \bigl( \Vert e_h \Vert + \Vert \nabla_{\mathcal M} e \Vert + \Vert \nabla_{\mathcal M} (x- I^k_h x)  \Vert \bigr) \nonumber  \\
& \leq & \bigl( \epsilon + c h^{\frac{1}{6}} \bigr) \Vert \nabla_{\mathcal M} e \Vert^2 + c_\epsilon \bigl( \Vert e_h \Vert^2 + h^{2k} \bigr).  \label{eq:S1est}
\end{eqnarray}
Next, recalling \eqref{eq:interpol} we have
\begin{eqnarray*}
S_2 &= &  \int_{\mathcal M} \underline D_i e \cdot \bigl( (\underline D_i P) e_h + P \underline D_i e + P \underline D_i (x- I^k_h x) \bigr) do \\
& \geq & \sum_{i=1}^3 \int_{\mathcal M} | P \underline D_i e |^2 do - c \Vert \nabla_{\mathcal M} e \Vert \Vert e_h \Vert - c h^k \Vert \nabla_{\mathcal M} e \Vert \\
& \geq &  \sum_{i=1}^3 \int_{\mathcal M} | P \underline D_i e |^2 do - \epsilon \Vert \nabla_{\mathcal M} e \Vert^2 - c_\epsilon \bigl( h^{2k} + \Vert e_h \Vert^2 \bigr).
\end{eqnarray*}
Finally, integration by parts together with \eqref{eq:tnuh} yields
\begin{eqnarray*}
| S_3 |  &= &  |  - \int_{\mathcal M} \Delta_{\mathcal M} x \cdot \bigl( (P_h - P) e_h \bigr) \, do | \leq c \int_{\mathcal M} | \hat \nu_h - \nu | \, | e_h | \, do  \leq c \Vert \hat \nu_h - \nu \Vert \Vert e_h \Vert \\
& \leq &  c \bigl( \Vert \nabla_{\mathcal M} e \Vert + h^k \bigr) \Vert e_h \Vert 
 \leq  \epsilon \Vert \nabla_{\mathcal M} e \Vert^2 + c_\epsilon \Vert e_h \Vert^2 + c h^{2k},
\end{eqnarray*}
so that returning to \eqref{eq:s1s2s3}
\begin{eqnarray}
\lefteqn{  \hspace{-1.5cm} 
\frac{1+ \alpha}{\alpha}  \int_{\mathcal M}  \bigl( \nabla_{\mathcal M}  x_h : \nabla_{\mathcal M}  \bigl( P_h e_h \bigr) - \nabla_{\mathcal M}  x :  \nabla_{\mathcal M} \bigl( P e_h \bigr) \bigr)  \, do }  \label{eq:err5} \\
& \geq & \frac{1+\alpha}{\alpha}   \sum_{i=1}^3 \int_{\mathcal M} | P \underline D_i e |^2 do - \bigl( \epsilon + c h^{\frac{1}{6}} \bigr) \Vert \nabla_{\mathcal M} e \Vert^2 - c_\epsilon \bigl( \Vert e_h \Vert^2 + h^{2k} \bigr).
\nonumber
\end{eqnarray}
If we insert \eqref{eq:e1}, \eqref{eq:err4} and \eqref{eq:err5} into \eqref{eq:err1} and estimate the terms $T_1,T_2,T_3$ in a straightforward way we obtain
\begin{eqnarray*}
\lefteqn{
 \frac{1}{2} \frac{d}{dt} \int_{\mathcal M} A(\nu,\sqrt{\dg}) e_h \cdot e_h   \, do +  c_0^2\sum_{i=1}^3 \int_{\mathcal M} (\underline D_i e \cdot \nu)^2 \, do  + \frac{1}{\alpha}   \sum_{i=1}^3 \int_{\mathcal M} | P \underline D_i e |^2 \, do } \\
 & \leq & \bigl( \epsilon+c h^{\frac{1}{6}} \bigr)  \Vert \nabla_{\mathcal M} e \Vert^2 + c \Vert e_h \Vert^2  + c_\epsilon h^{2k} 
 + c \Vert x_t - I^k_h x_t  \Vert \, \Vert e_h \Vert \\
 & &  + c \Vert x_{h,t} \Vert_{L^\infty}  \Vert e_h \Vert \, \Vert \nabla_{\mathcal M} e \Vert + c \Vert \nabla_{\mathcal M} e \Vert \, \Vert \nabla_{\mathcal M} (x - I^k_h x) \Vert \\
 & \leq &  \bigl( \epsilon+c h^{\frac{1}{6}} \bigr)  \Vert \nabla_{\mathcal M} e \Vert^2 + c_\epsilon \bigl( 1+ \Vert x_{h,t} \Vert_{L^\infty}^2 \bigr) \Vert e_h \Vert^2 + c_\epsilon h^{2k}.
 \end{eqnarray*}
 Note that 
 \begin{equation} \label{eq:h1est}
  c_0^2  \sum_{i=1}^3 \int_{\mathcal M} (\underline D_i e \cdot \nu)^2 \, do  + \frac{1}{\alpha}  \sum_{i=1}^3 \int_{\mathcal M} | P \underline D_i e |^2 \, do \geq
   \min(c_0^2,\frac{1}{\alpha}) \Vert \nabla_{\mathcal M} e \Vert^2.
  \end{equation}
  Choosing $\epsilon>0, h_0>0$ sufficiently small we thus obtain after integration with respect to time with the help of \eqref{eq:Hpos}
 \begin{displaymath}
 \Vert e_h(t) \Vert^2 +  \int_0^t \Vert \nabla_{\mathcal M} e (s) \Vert^2 ds \leq c h^{2k} + c \int_0^t  \bigl( 1+ \Vert x_{h,t}(s) \Vert_{L^\infty}^2 \bigr) \Vert e_h(s) \Vert^2 \, ds, \quad 0 \leq t  \leq \hat T_h
 \end{displaymath}
 and Gronwall's lemma implies the first two bounds in \eqref{eq:esth1} recalling that $\int_0^{\hat T_h}  \Vert x_{h,t}(t) \Vert^2_{L^\infty} dt \leq 2c_1$. In order to prove the third estimate in \eqref{eq:esth1}
 we argue in a similar way as in \eqref{eq:este1} and obtain
 \begin{displaymath}
 \int_0^{\hat T_h} \Vert \nabla_{\mathcal M} e \Vert_{L^\infty}^2 dt \leq  \hat c h^{-2} \int_0^{\hat T_h}  \bigl( \Vert \nabla_{\mathcal M} e \Vert^2 + \Vert \nabla_{\mathcal M} (x-I^k_h x) \Vert^2 \bigr) dt + c h^{2k} \leq c h^{2k-2}
 \end{displaymath}
 where $c$ only depends on $x, \hat c$ and $T$.
 \end{proof}
 
 \noindent
 In the second part of the analysis we aim to bound $\displaystyle \max_{0 \leq t \leq \hat T_h} \Vert \nabla_{\mathcal M} e \Vert^2 + \int_0^{\hat T_h} \Vert e_t \Vert^2 dt$. We start by
 inserting  $\chi_h = e_{h,t}$ into \eqref{eq:erroreq}:
\begin{eqnarray}
\lefteqn{
\int_{\mathcal M} A(\nu,\sqrt{\dg})  e_{h,t} \cdot e_{h,t} \, do  + \int_{\mathcal M} \bigl( G_h^{-1} \nabla_{\mathcal M} x_h \sqrt{\dgh} - G^{-1} \nabla_{\mathcal M}  x \sqrt{\dg} \bigr) : \nabla_{\mathcal M} e_t   \, do } \nonumber \\
& & + \frac{1+\alpha}{\alpha}  \int_{\mathcal M}  \bigl( \nabla_{\mathcal M} x_h : \nabla_{\mathcal M} (P_h e_{h,t}) - \nabla_{\mathcal M}   x :  \nabla_{\mathcal M}  ( P e_{h,t}) \bigr)  \, do \label{eq:err6}  \\
& = &  \int_{\mathcal M} \bigl( G_h^{-1} \nabla_{\mathcal M}  x_h \sqrt{\dgh} - G^{-1} \nabla_{\mathcal M} x \sqrt{\dg} \bigr) : \nabla_{\mathcal M} (I^k_h x_t - x_t)  \, do \nonumber  \\
& &  +  \int_{\mathcal M}  \bigl( A(\nu,\sqrt{\dg}) - A(\nu_h,\sqrt{\dgh}) \bigr) x_{h,t} \cdot e_{h,t}  \, do +  \int_{\mathcal M} A(\nu,\sqrt{\dg}) (x_t - I^k_h x_t)  \cdot e_{h,t}   \, do. \nonumber
\end{eqnarray}

\begin{lem} We have
\begin{eqnarray}
\lefteqn{ \int_{\mathcal M} 
\bigl( G_h^{-1} \nabla_{\mathcal M} x_h \sqrt{\dgh} - G^{-1} \nabla_{\mathcal M}  x \sqrt{\dg} \bigr) : \nabla_{\mathcal M} e_t \, do  } \label{eq:ghg}  \\
& \geq  &  \frac{d}{dt} \int_{\mathcal M}  \bigl( \sqrt{\dgh} + \sqrt{\dg} - G^{-1} \nabla_{\mathcal M} x  :  \nabla_{\mathcal M} x_h  \sqrt{\dg}\bigr) \, do 
- c \int_{\mathcal M} | \nabla_{\mathcal M} e |^2 \, do. \nonumber
\end{eqnarray}
\end{lem}
\begin{proof} To begin, recalling that $e=x_h-x$ and using \eqref{eq:dtdet} for $G$ and $G_h$ we have
\begin{eqnarray}
\lefteqn{
\bigl( G_h^{-1} \nabla_{\mathcal M} x_h \sqrt{\dgh} - G^{-1} \nabla_{\mathcal M}  x \sqrt{\dg} \bigr) : \nabla_{\mathcal M} (x_{h,t} - x_t) } \label{eq:err8}  \\
& = &  G_h^{-1} \nabla_{\mathcal M}  x_h  :  \nabla_{\mathcal M}  x_{h,t}  \sqrt{\dgh} - G_h^{-1} \nabla_{\mathcal M}  x_h  : \nabla_{\mathcal M}  x_t  \sqrt{\dgh} \nonumber \\
& & - \partial_t \bigl( G^{-1} \nabla_{\mathcal M} x  :  \nabla_{\mathcal M}  x_h  \sqrt{\dg}\bigr)  + (\partial_t G^{-1}) \nabla_{\mathcal M}  x :  \nabla_{\mathcal M}  x_h \sqrt{\dg} + G^{-1}
 \nabla_{\mathcal M}  x :  \nabla_{\mathcal M} x_h \partial_t \sqrt{\dg} 
\nonumber \\
& &  + G^{-1} \nabla_{\mathcal M}  x_t  :  \nabla_{\mathcal M}  x_h  \sqrt{\dg}  + G^{-1} \nabla_{\mathcal M}  x :  \nabla_{\mathcal M}  x_t 
\sqrt{\dg} \nonumber \\
& = & \partial_t \bigl( \sqrt{\dgh} + \sqrt{\dg} - G^{-1} \nabla_{\mathcal M}  x  :  \nabla_{\mathcal M}  x_h  \sqrt{\dg}\bigr)  + (\partial_t G^{-1}) \nabla_{\mathcal M}  x :  \nabla_{\mathcal M}  x_h \sqrt{\dg} 
\nonumber \\
& & + G^{-1} \nabla_{\mathcal M}  x :  \nabla_{\mathcal M} x_h \partial_t \sqrt{\dg}  
  + G^{-1} \nabla_{\mathcal M}  x_t  : \nabla_{\mathcal M}  x_h  \sqrt{\dg} - G_h^{-1} \nabla_{\mathcal M}  x_t :  \nabla_{\mathcal M} x_h  \sqrt{\dgh}, \nonumber
\end{eqnarray}
where we also used the symmetry of $G^{-1}_h$. We infer with the help of \eqref{eq:ghinvexp} and \eqref{eq:detexp} 
\begin{eqnarray*} 
\lefteqn{
G^{-1} \nabla_{\mathcal M}  x_t  : \nabla_{\mathcal M}  x_h  \sqrt{\dg} - G_h^{-1} \nabla_{\mathcal M}  x_t :  \nabla_{\mathcal M} x_h  \sqrt{\dgh} } \\
& = & G^{-1}(G_h - G) G^{-1} \nabla_{\mathcal M} x_t : \nabla_{\mathcal M} x_h \sqrt{\dg} - (G^{-1} \nabla_{\mathcal M}x : \nabla_{\mathcal M} e) G^{-1} \nabla_{\mathcal M} x_t : \nabla_{\mathcal M} x_h
\sqrt{\dg} + r^{(3)}_h \\
& = &  G^{-1}(G_h - G) G^{-1} \nabla_{\mathcal M} x_t : \nabla_{\mathcal M} x \sqrt{\dg} - (G^{-1} \nabla_{\mathcal M}x : \nabla_{\mathcal M} e) G^{-1} \nabla_{\mathcal M} x_t : \nabla_{\mathcal M} x
\sqrt{\dg} + r^{(4)}_h \\
& = & G^{-1}G_h  G^{-1} \nabla_{\mathcal M} x_t : \nabla_{\mathcal M} x \sqrt{\dg} - G^{-1} \nabla_{\mathcal M} x_t : \nabla_{\mathcal M} x \sqrt{\dg} 
- (G^{-1} \nabla_{\mathcal M} x: \nabla_{\mathcal M} e) \partial_t \sqrt{\dg} + r^{(4)}_h,
\end{eqnarray*}
where $| r^{(3)}_h |, | r^{(4)}_h | \leq c | \nabla_{\mathcal M} e |^2$. Inserting this relation into \eqref{eq:err8} we obtain
\begin{eqnarray}
\lefteqn{ 
\bigl( G_h^{-1} \nabla_{\mathcal M} x_h \sqrt{\dgh} - G^{-1} \nabla_{\mathcal M}  x \sqrt{\dg} \bigr) : \nabla_{\mathcal M} (x_{h,t} - x_t) } \label{eq:err9}  \\
& = & \partial_t \bigl( \sqrt{\dgh} + \sqrt{\dg} - G^{-1} \nabla_{\mathcal M}  x  :  \nabla_{\mathcal M}  x_h  \sqrt{\dg}\bigr) + (\partial_t G^{-1}) \nabla_{\mathcal M}  x :  \nabla_{\mathcal M}  x_h \sqrt{\dg} 
\nonumber \\
& & + G^{-1} \nabla_{\mathcal M}  x :  \nabla_{\mathcal M} x_h \partial_t \sqrt{\dg} +G^{-1}G_h  G^{-1} \nabla_{\mathcal M} x_t : \nabla_{\mathcal M} x \sqrt{\dg} \nonumber  \\
& & - G^{-1} \nabla_{\mathcal M} x_t : \nabla_{\mathcal M} x \sqrt{\dg} 
- (G^{-1} \nabla_{\mathcal M}x : \nabla_{\mathcal M} e) \partial_t \sqrt{\dg} + r^{(4)}_h  \nonumber\\
& = & \partial_t \bigl( \sqrt{\dgh} + \sqrt{\dg} - G^{-1} \nabla_{\mathcal M}  x  :  \nabla_{\mathcal M}  x_h  \sqrt{\dg}\bigr) + (\partial_t G^{-1}) \nabla_{\mathcal M}  x :  \nabla_{\mathcal M}  x_h \sqrt{\dg} 
\nonumber \\
& & + G^{-1} \nabla_{\mathcal M}  x :  \nabla_{\mathcal M} x \partial_t \sqrt{\dg} +G^{-1}G_h  G^{-1} \nabla_{\mathcal M} x_t : \nabla_{\mathcal M} x \sqrt{\dg}
- G^{-1} \nabla_{\mathcal M} x_t : \nabla_{\mathcal M} x \sqrt{\dg} + r^{(4)}_h \nonumber \\
& = & \partial_t \bigl( \sqrt{\dgh} + \sqrt{\dg} - G^{-1} \nabla_{\mathcal M}  x  :  \nabla_{\mathcal M}  x_h  \sqrt{\dg}\bigr) + (\partial_t G^{-1}) \nabla_{\mathcal M}  x :  \nabla_{\mathcal M}  x_h \sqrt{\dg} 
\nonumber \\
& &  +G^{-1}G_h  G^{-1} \nabla_{\mathcal M} x_t : \nabla_{\mathcal M} x \sqrt{\dg}
+G^{-1} \nabla_{\mathcal M} x_t : \nabla_{\mathcal M} x \sqrt{\dg} + r^{(4)}_h \nonumber 
\end{eqnarray}
where we used \eqref{eq:trace} and \eqref{eq:dtdet} in the last step. We have for the
components of $\partial_t G^{-1}$
\begin{displaymath}
\partial_t G^{ij} = - G^{ik} G^{j \ell} \bigl( \underline D_k x_t \cdot \underline D_\ell x + \underline D_k x \cdot \underline D_\ell x_t \bigr), \quad 1 \leq i,j \leq 3,
\end{displaymath}
and therefore in view of the symmetry of $G^{-1}$
\begin{eqnarray*}
\lefteqn{ (\partial_t G^{-1}) \nabla_{\mathcal M}  x :  \nabla_{\mathcal M}  x_h = 
\partial_t G^{ij} \underline D_j x \cdot \underline D_i x_h } \\
&=& - G^{ik} G^{j \ell} \bigl( \underline D_k x_t \cdot \underline D_\ell x + \underline D_k x \cdot \underline D_\ell x_t \bigr) 
\underline D_j x \cdot \underline D_i x_h  
 =  - G^{ik} G^{j \ell} \underline D_k x \cdot \underline D_\ell x_t \bigl( \underline D_i x \cdot \underline D_j x_h + \underline D_j x \cdot \underline D_i x_h \bigr).
\end{eqnarray*}
Furthermore,
\begin{eqnarray*}
\lefteqn{
G^{-1}G_h  G^{-1} \nabla_{\mathcal M} x_t : \nabla_{\mathcal M} x =  G^{ki} G_{h,ij} G^{j \ell}  \underline D_\ell x_t \cdot \underline D_k x  } \\
& = & G^{ik} G^{j \ell} \bigl( \underline D_i x_h \cdot \underline D_j x_h + \mu_i \mu_j \bigr) \underline D_k x \cdot \underline D_\ell x_t = 
G^{ik} G^{j \ell}  (\underline D_i x_h \cdot \underline D_j x_h)  (\underline D_k x \cdot \underline D_\ell x_t ),
\end{eqnarray*}
since $G^{ik} G^{j \ell}  \mu_i \mu_j \underline D_k x \cdot \underline D_\ell x_t = \mu_k \mu_\ell \underline D_k x \cdot \underline D_\ell x_t =0$ in view of \eqref{eq:trace}.  Combining the above two relations
we deduce that
\begin{eqnarray*}
\lefteqn{
(\partial_t G^{-1}) \nabla_{\mathcal M}  x :  \nabla_{\mathcal M}  x_h   +G^{-1}G_h  G^{-1} \nabla_{\mathcal M} x_t : \nabla_{\mathcal M} x  } \\
& = &  - G^{ik} G^{j \ell} \underline (D_k x \cdot \underline D_\ell x_t) \bigl( \underline D_i x \cdot \underline D_j x_h + \underline D_j x \cdot \underline D_i x_h \bigr) 
+ G^{ik} G^{j \ell}  (\underline D_i x_h \cdot \underline D_j x_h)  (\underline D_k x \cdot \underline D_\ell x_t) \\
& = & G^{ik} G^{j \ell} \underline D_i(x_h -x) \cdot \underline D_j(x_h-x) (\underline D_k x \cdot \underline D_\ell x_t) - G^{ik} G^{j \ell} (\underline D_i x \cdot \underline D_j x) (\underline D_k x \cdot \underline D_\ell x_t) \\
& = & G^{ik} G^{j \ell} (\underline D_i e \cdot \underline D_j e) (\underline D_k x \cdot \underline D_\ell x_t) - G^{ik} G^{j \ell} \bigl( G_{ij} - \mu_i \mu_j \bigr) (\underline D_k x \cdot \underline D_\ell x_t) \\
& = & G^{ik} G^{j \ell} (\underline D_i e \cdot \underline D_j e) (\underline D_k x \cdot \underline D_\ell x_t) - G^{k \ell}  (\underline D_k x \cdot \underline D_\ell x_t) \\
& = & G^{ik} G^{j \ell} (\underline D_i e \cdot \underline D_j e) (\underline D_k x \cdot \underline D_\ell x_t) - G^{-1} \nabla_{\mathcal M} x: \nabla_{\mathcal M} x_t.
\end{eqnarray*} 
If we insert this identity  into \eqref{eq:err9} we obtain
\begin{eqnarray*}
\lefteqn{ 
\bigl( G_h^{-1} \nabla_{\mathcal M} x_h \sqrt{\dgh} - G^{-1} \nabla_{\mathcal M}  x \sqrt{\dg} \bigr) : \nabla_{\mathcal M} (x_{h,t} - x_t) } \\  \label{eq:err7}
& = &  \partial_t \bigl( \sqrt{\dgh} + \sqrt{\dg} - G^{-1} \nabla_{\mathcal M}  x  :  \nabla_{\mathcal M}  x_h  \sqrt{\dg}\bigr) + r^{(5)}_h, \nonumber
\end{eqnarray*}
where $| r^{(5)}_h | \leq c | \nabla_{\mathcal M} e |^2$. Integration of the above relation over $\mathcal M$ finishes the proof.
\end{proof}

\begin{lem}
\begin{eqnarray}
\lefteqn{ 
\int_{\mathcal M}  \bigl( \nabla_{\mathcal M}  x_h : \nabla_{\mathcal M}  \bigl( P_h e_{h,t} \bigr) - \nabla_{\mathcal M} x : \nabla_{\mathcal M}  \bigl( P e_{h,t} \bigr) \bigr)  \, do }  \label{eq:err10} \\
&  \geq &  \frac{1}{2}  \frac{d}{dt} \sum_{i=1}^3 \int_{\mathcal M} | P  \underline D_i  e |^2 \, do - \epsilon \Vert e_{h,t} \Vert^2 - c_\epsilon \bigl( 1+ h^{-2}  \Vert \nabla_{\mathcal M} e \Vert_{L^\infty}^2 \bigr)
\bigl( \Vert \nabla_{\mathcal M} e \Vert^2 + h^{2k} \bigr). \nonumber
\end{eqnarray}
\end{lem}
\begin{proof} Let us write
\begin{eqnarray*}
\lefteqn{ \hspace{-1.5cm}
\int_{\mathcal M}  \bigl( \nabla_{\mathcal M}  x_h : \nabla_{\mathcal M}  \bigl( P_h e_{h,t} \bigr) - \nabla_{\mathcal M}  x : \nabla_{\mathcal M}  \bigl( P  e_{h,t} \bigr) \bigr)  \, do 
 =  \int_{\mathcal M} \nabla_{\mathcal M} e : \nabla_{\mathcal M}  \bigl( (P_h-P) e_{h,t} \bigr) \, do } \\
 & & + \int_{\mathcal M} \nabla_{\mathcal M}  e : \nabla_{\mathcal M}  \bigl( P e_{h,t} \bigr) \, do + \int_{\mathcal M} \nabla_{\mathcal M}  x : \nabla_{\mathcal M}  \bigl( (P_h - P) e_{h,t} \bigr) \, do \\
& \equiv : & S_1+ S_2 + S_3.
\end{eqnarray*}
Using \eqref{eq:phmp} with $e_{h,t}$ instead of $e_h$ we derive with the help of \eqref{eq:inverse} and \eqref{eq:tnuh1}
\begin{eqnarray*}
| S_1 | & \leq & c \int_{\mathcal M} | \nabla_{\mathcal M} e| \bigl(  | e_{h,t} | \, | \nu - \hat \nu_h| + | \nabla_{\mathcal M} e_{h,t}  | \, | \nu - \hat \nu_h | + |e_{h,t} | | \nabla_{\mathcal M} (\nu - \hat \nu_h) | \bigr) do \\
& \leq & c \Vert \nabla_{\mathcal M} e \Vert_{L^\infty} \bigl( \Vert \hat \nu_h - \nu \Vert \,  \Vert e_{h,t} \Vert_{H^1}  + \Vert \nabla_{\mathcal M} (\hat \nu_h - \nu)  \Vert \,  \Vert e_{h,t} \Vert \bigr) \\
& \leq & c h^{-1}  \Vert \nabla_{\mathcal M} e \Vert_{L^\infty}  \bigl(\Vert \nabla_{\mathcal M} e \Vert + h^k\bigr)  \Vert e_{h,t} \Vert  
 \leq  \epsilon \Vert e_{h,t} \Vert^2 + c_\epsilon  h^{-2}  \Vert \nabla_{\mathcal M} e \Vert_{L^\infty}^2 \bigl( \Vert \nabla_{\mathcal M} e \Vert^2 + h^{2k} \bigr).
\end{eqnarray*}
Next,
\begin{eqnarray*}
S_2 &= &  \int_{\mathcal M} \underline D_i  e \cdot \bigl( (\underline D_i P) e_{h,t} + P \underline D_i e_t + P \underline D_i (x_t - I^k_h x_t) \bigr) do \\
& \geq & \frac{1}{2} \frac{d}{dt} \int_{\mathcal M} \underline D_i e \cdot P \underline D_i e \, do - c  \Vert \nabla_{\mathcal M} e \Vert \bigl( \Vert e_{h,t} \Vert + \Vert 
\nabla_{\mathcal M} e \Vert + h^k \bigr) \\
& \geq & \frac{1}{2} \frac{d}{dt} \int_{\mathcal M} | P \underline D_i e |^2 \, do  - \epsilon \Vert e_{h,t} \Vert^2 - c_\epsilon \Vert \nabla_{\mathcal M} e \Vert^2 - c h^{2k}.
\end{eqnarray*}
Finally, integration by parts together with \eqref{eq:tnuh1} yields
\begin{eqnarray*}
| S_3 |  &= &  |  - \int_{\mathcal M} \Delta_{\mathcal M} x \cdot \bigl( (P_h - P) e_{h,t} \bigr) \, do | \leq c \int_{\mathcal M} | \hat \nu_h - \nu | \, |e_{h,t} | \, do 
 \leq  c \Vert \hat \nu_h - \nu \Vert \Vert e_{h,t} \Vert  \\
 &\leq &  c \bigl( \Vert \nabla_{\mathcal M} e \Vert + h^k \bigr) \Vert e_{h,t} \Vert 
 \leq  \epsilon  \Vert e_{h,t} \Vert^2 + c_\epsilon \Vert \nabla_{\mathcal M} e \Vert^2 + c_\epsilon h^{2k}.
\end{eqnarray*}
The above estimates imply the assertion of the lemma.
\end{proof}

\noindent
If we insert  \eqref{eq:ghg}  and \eqref{eq:err10} into \eqref{eq:err6}, recall \eqref{eq:Hpos}, \eqref{eq:prop1} and estimate the integrals on the right hand side in a straightforward manner we obtain
\begin{eqnarray}
\lefteqn{ \hspace{-8mm} 
\min(1,c_0)  \Vert e_{h,t} \Vert^2  
+\frac{1+\alpha}{2 \alpha } \frac{d}{dt}  \sum_{i=1}^3 \int_{\mathcal M} | P \underline D_i e |^2 \, do } \nonumber  \\
& & +  \frac{d}{dt} \int_{\mathcal M}  \bigl( \sqrt{\dgh} + \sqrt{\dg} - G^{-1} \nabla_{\mathcal M} x : \nabla_{\mathcal M}  x_h  \sqrt{\dg}\bigr) \, do \nonumber \\
& \leq &  \epsilon \Vert e_{h,t} \Vert^2 +   c_\epsilon \bigl( 1+ h^{-2}  \Vert \nabla_{\mathcal M} e \Vert_{L^\infty}^2 \bigr)
\bigl( \Vert \nabla_{\mathcal M} e \Vert^2 + h^{2k} \bigr) 
 + c \Vert \nabla_{\mathcal M} e \Vert \, \Vert \nabla_{\mathcal M} (x_t - I^k_h x_t)  \Vert \nonumber \\
 & & + c \Vert x_{h,t} \Vert_{L^\infty} \Vert e_{h,t} \Vert \, \Vert \nabla_{\mathcal M} e \Vert + c \Vert x_t - I^k_h x_t  \Vert \, \Vert e_{h,t} \Vert \nonumber \\
& \leq &  \epsilon \Vert e_{h,t} \Vert^2 +   c_\epsilon \bigl( 1 + h^{-2}  \Vert \nabla_{\mathcal M} e \Vert_{L^\infty}^2 \bigr) 
\bigl( \Vert \nabla_{\mathcal M} e \Vert^2 + h^{2k} \bigr)  + c_\epsilon \Vert x_{h,t} \Vert_{L^\infty}^2 \Vert \nabla_{\mathcal M} e \Vert^2. \label{eq:err11}
\end{eqnarray}

\noindent
Before we integrate this relation with respect to time we have a closer look at the third  term on the left hand side.
\begin{lem} 
\begin{eqnarray}
\lefteqn{ \int_{\mathcal M} \bigl(
\sqrt{\dgh} + \sqrt{\dg} - G^{-1} \nabla_{\mathcal M}  x  : \nabla_{\mathcal M}  x_h  \sqrt{\dg} \bigr) \, do }\label{eq:err12}  \\
& \geq  & \frac{c_0^2}{2} \sum_{i=1}^3 \int_{\mathcal M}  ( \underline D_i e \cdot \nu)^2 \, do  - \frac{1}{2}  \sum_{i=1}^3 \int_{\mathcal M} | P \underline D_i e |^2 \, do  - 
c \int_{\mathcal M} | \nabla_{\mathcal M} e |^3 \, do. \nonumber
\end{eqnarray}
\end{lem}
\begin{proof} Let us define $G(\rho)$ by $G_{ij}(\rho):=\underline D_i(x+\rho e) \cdot \underline D_j(x+ \rho e)+\mu_i \mu_j$  as well as $f(\rho):= \sqrt{\mbox{det} \, G(\rho)}$. 
We have
\begin{displaymath}
f'(\rho)  =  G^{-1}(\rho) \nabla_{\mathcal M} e : \nabla_{\mathcal M} (x+ \rho e) \sqrt{\mbox{det} \, G(\rho)}
\end{displaymath}
so that using \eqref{eq:trace}
\begin{displaymath}
 f'(0)  =  G^{-1}  \nabla_{\mathcal M} e :  \nabla_{\mathcal M}  x \sqrt{\dg} = G^{-1}  \nabla_{\mathcal M}  x_h : \nabla_{\mathcal M}  x \sqrt{\dg} - 2 \sqrt{\dg}.
 \end{displaymath}
 Thus we may write with the help of a Taylor expansion
\begin{equation} \label{eq:lowest1}
\sqrt{\dgh} + \sqrt{\dg} - G^{-1} \nabla_{\mathcal M}  x : \nabla_{\mathcal M}  x_h  \sqrt{\dg} = f(1) - f(0) - f'(0) = \frac{1}{2} f''(0) + \frac{1}{6} f'''(c),
\end{equation}
for some $c \in [0,1]$.  Next we have
\begin{eqnarray*}
f''(\rho) & = &  \frac{d}{d \rho} G^{-1}(\rho) \nabla_{\mathcal M} e : \nabla_{\mathcal M} (x+ \rho e) \sqrt{\mbox{det} \, G(\rho)} 
+ G^{-1}(\rho) \nabla_{\mathcal M}  e : \nabla_{\mathcal M}  e \, \sqrt{\mbox{det} \, G(\rho)} \\
& & + \bigl( G^{-1}(\rho)  \nabla_{\mathcal M} e : \nabla_{\mathcal M} (x+ \rho e) \bigr) \bigl( G^{-1}(\rho)  \nabla_{\mathcal M} e : \nabla_{\mathcal M} (x+ \rho e) \bigr) \sqrt{\mbox{det} \, G(\rho)}.
\end{eqnarray*}
Since $\frac{d}{d \rho} G^{-1}(\rho)=-G^{-1}(\rho) \frac{d}{d \rho} G(\rho) G^{-1}(\rho)$ we obtain with the help of  \eqref{eq:decompe} 
\begin{eqnarray*}
f''(0) & = &  -  G^{j \ell} \bigl(  \underline D_\ell e \cdot \underline D_k x  +  \underline D_\ell x   \cdot \underline D_k e \bigr) G^{ki} \underline D_i e \cdot \underline D_j x  \sqrt{\dg} + G^{ik} \underline D_i e \cdot \underline D_k e \sqrt{\dg} \\
 & & + G^{ik}  G^{j \ell} (\underline D_i  e \cdot \underline D_k x) \, (\underline D_j x \cdot D_\ell e)  \sqrt{\dg} \\
 & = &  - G^{ik} G^{j \ell} \bigl(  \underline D_\ell e \cdot \underline D_k x  +  \underline D_\ell x   \cdot \underline D_k e \bigr)  \underline D_i e \cdot \underline D_j x  \sqrt{\dg} \\
 & & + G^{ik} G^{j \ell}  (\underline D_i e \cdot \underline D_j x)(\underline D_k e \cdot \underline D_\ell x) \sqrt{\dg}  + G^{ik} (\underline D_i e \cdot \nu) (\underline D_k e \cdot \nu) \sqrt{\dg} \\
 & & + G^{ik}  G^{j \ell} (\underline D_i  e \cdot \underline D_k x)  \, (\underline D_j x \cdot D_\ell e)  \sqrt{\dg} \\
 & = &  G^{ik} G^{j \ell} (\underline D_i e \cdot \underline D_k x)  \, (\underline D_j x \cdot D_\ell e) \sqrt{\mbox{det} G}
- G^{ik} G^{j \ell}   (\underline D_k x \cdot \underline D_\ell e) \, (\underline D_i e \cdot \underline D_j x) \sqrt{\mbox{det} G} \\
& & + G^{ik} (\underline D_i e \cdot \nu) (\underline D_k e \cdot \nu) \sqrt{\mbox{det} G} \\
& = & G^{ik} G^{j \ell}  \bigl( \underline D_j x \wedge \underline D_k x \bigr)  \cdot \bigl( \underline D_\ell e \wedge \underline D_i e \bigr)\sqrt{\mbox{det} G}  + G^{ik} (\underline D_i e \cdot \nu) (\underline D_k e \cdot \nu) \sqrt{\mbox{det} G}.
\end{eqnarray*}
\noindent
Inserting this relation into \eqref{eq:lowest1}, using \eqref{eq:estpe},  \eqref{eq:prop1} as well as the fact that $| f'''(c) | \leq c | \nabla_{\mathcal M} e |^3$ we obtain
\begin{eqnarray*}
\lefteqn{
\sqrt{\dgh} + \sqrt{\dg} - G^{-1} \nabla_{\mathcal M}  x : \nabla_{\mathcal M}  x_h  \sqrt{\dg} } \\
& \geq & \frac{1}{2}  G^{ij} (\underline D_i e \cdot \nu) (\underline D_j e \cdot \nu) \sqrt{\dg}  - \frac{1}{2}  | G^{ik} G^{j \ell} \bigl(\underline D_j x \wedge \underline D_k x \bigr)  \cdot \bigl( \underline D_\ell e \wedge \underline D_i e \bigr) | \sqrt{\mbox{det} G}  - c | \nabla_{\mathcal M} e |^3  \\
& \geq & \frac{c_0^2}{2}  \sum_{i=1}^3 ( \underline D_i e \cdot \nu)^2 - \frac{1}{2}  \sum_{i=1}^3  | P  \underline D_i e |^2 - c | \nabla_{\mathcal M} e |^3
\end{eqnarray*}
and the result follows after integration over $\mathcal M$. 
\end{proof}

\noindent
Integrating \eqref{eq:err11} with respect to time and using \eqref{eq:err12} we obtain after choosing $\epsilon$ sufficiently small
\begin{eqnarray*}
\lefteqn{
\frac{1}{2} \min(1,c_0)  \int_0^t \Vert e_{h,t} \Vert^2 ds + \frac{c_0^2}{2}   \sum_{i=1}^3 \int_{\mathcal M}  ( \underline D_i e  \cdot \nu)(t)^2 \, do + \frac{1}{2 \alpha} 
 \sum_{i=1}^3 \int_{\mathcal M} | (P \underline D_i e)(t) |^2 \, do   } \\
& \leq & c \int_{\mathcal M} | \nabla_{\mathcal M} e(t) |^3 \, do +  c \Vert \nabla_{\mathcal M} e(0) \Vert^2  \\
& & +  c \int_0^t  \left\{  \bigl( 1+ h^{-2}  \Vert \nabla_{\mathcal M} e \Vert_{L^\infty}^2 \bigr)\bigl( \Vert \nabla_{\mathcal M} e \Vert^2 + h^{2k} \bigr) + \Vert x_{h,t} \Vert _{L^\infty}^2
\Vert \nabla_{\mathcal M} e \Vert^2 \right\} \, ds \\
& \leq &  c h^{\frac{1}{2}}  \Vert \nabla_{\mathcal M} e(t) \Vert^2 + c h^{2k}  \bigl( 1 + \int_0^{\hat T_h} h^{-2}  \Vert \nabla_{\mathcal M} e \Vert_{L^\infty}^2 dt \bigr) \\
 & & + c  \int_0^t   \bigl( 1+ h^{-2}  \Vert \nabla_{\mathcal M} e \Vert_{L^\infty}^2 + \Vert x_{h,t} \Vert_{L^\infty}^2 \bigr) \Vert \nabla_{\mathcal M} e \Vert^2 ds,
\end{eqnarray*}
where we used \eqref{eq:este1} in the last step. Note that \eqref{eq:esth1} implies that 
\begin{equation} \label{eq:estint}
h^{-2} \int_0^{\hat T_h} \Vert \nabla_{\mathcal M} e(t) \Vert_{L^\infty}^2 dt \leq c_2 h^{2k-4} \leq c_2
\end{equation}
since we have assumed that $k \geq 2$. 
Recalling \eqref{eq:h1est} and absorbing the first term on the right hand side for sufficiently small $h_0$ we deduce that 
\begin{displaymath}
\int_0^t \Vert e_{h,t} \Vert^2 ds + \Vert \nabla_{\mathcal M} e(t) \Vert^2 \leq ch^{2k} + c  \int_0^t   \bigl( 1+ h^{-2}  \Vert \nabla_{\mathcal M} e \Vert_{L^\infty}^2 + \Vert x_{h,t} \Vert_{L^\infty}^2  \bigr) \Vert \nabla_{\mathcal M} e \Vert^2 ds.
\end{displaymath}
 Gronwall's lemma along with \eqref{eq:estint} and the bound  $\int_0^{\hat T_h}  \Vert x_{h,t}(t) \Vert^2_{L^\infty} dt \leq 2c_1$ then yields
\begin{displaymath}
\int_0^{\hat T_h} \Vert e_{h,t} \Vert^2 ds + \max_{0 \leq t \leq \hat T_h} \Vert \nabla_{\mathcal M} e(t) \Vert^2 \leq c h^{2k},
\end{displaymath}
where the constant $c$ is independent of $h$. In particular, 
$\Vert \nabla_{\mathcal M} e(\hat T_h) \Vert  \leq c h^k \leq c h^2 \leq \frac{1}{2} h^{\frac{3}{2}}, 0<h \leq h_0$ after decreasing $h_0$ further if necessary. In addition we infer from \eqref{eq:prop2} and \eqref{eq:inverse} that
\begin{eqnarray*}
\lefteqn{
\int_0^{\hat T_h} \Vert x_{h,t} \Vert_{L^\infty}^2 dt \leq   \int_0^{\hat T_h} \bigl( \Vert x_t \Vert_{L^\infty} + \Vert x_t - I^k_h x_t  \Vert_{L^\infty}+ \Vert e_{h,t} \Vert_{L^\infty} \bigr)^2 dt }  \\
& \leq & \frac{4}{3} \int_0^T \Vert x_t \Vert_{L^\infty}^2 dt + c h^{2k} + c h^{-2}  \int_0^{\hat T_h} \Vert e_{h,t} \Vert^2 dt \leq \frac{4}{3} c_1 + c h^{2k-2} \leq \frac{3}{2} c_1
\end{eqnarray*}
by choosing $h_0$ smaller if necessary.
Suppose that  $\hat T_h<T$. In view of the above estimates  we could then extend the solution $x_h$ to an interval
$[0,\hat T_h + \delta]$ such that $\Vert \nabla_{\mathcal M}  e(t) \Vert  \leq h^{\frac{3}{2}}$ for $\hat T_h \leq t \leq \hat T_h + \delta$ and $\int_0^{\hat T_h+\delta} \Vert x_{h,t} \Vert_{L^\infty}^2 dt \leq 2 c_1$ contradicting
the definition of $\hat T_h$. As a result, $\hat T_h=T$ and Theorem \ref{thm:main} is proved.

\section{Numerical results} \label{sec:5}
\setcounter{equation}{0}

\subsection{Finite element spaces and fully discrete schemes} \label{sec:5.1}
\noindent
Let us begin by describing the space $S^k_h$, which we have used for implementing our scheme.  We denote by $d$ the oriented distance function for the reference surface $\mathcal M$, for which there exists $\delta>0$ such
that $d$ is smooth in $U:= \lbrace p \in \mathbb R^3 \, | \, | d(p) | < \delta \rbrace$ and each $p \in U$ has a unique decomposition
\begin{displaymath}
p=a(p)+d(p) \mu(p)
\end{displaymath}
where $a(p)$ is the closest point projection onto $\mathcal M$ and $\mu(p)=\nabla d(p)$. 
 For $0<h \leq h_0$ we choose a polyhedral approximation $\mathcal M_h$ of $\mathcal M$ consisting of triangles $\tilde T$ of diameter $h$ that are
 shape regular and quasiuniform with vertices lying
on $\mathcal M$ and assume that $a_{| \mathcal M_h}: \mathcal M_h \rightarrow \mathcal M$ is bijective. 
Furthermore, let
\begin{displaymath}
\tilde S^k_h:= \lbrace \tilde \chi_h \in C^0(\mathcal M_h) \, | \,  \tilde \chi_{h| \tilde T} \mbox{ is a polynomial of degree} \leq k \mbox{ for each } \tilde T \subset \mathcal M_h \rbrace
\end{displaymath}
and then define
\begin{displaymath}
S^k_h:= \lbrace \chi_h \in C^0(\mathcal M) \, | \, \chi_h \circ a \in \tilde S^k_h \rbrace \subset W^{1,\infty}(\mathcal M).
\end{displaymath}
The interpolation operator $I^k_h:C^0(\mathcal M) \rightarrow S^k_h$ is then defined as $(I^k_h f)(p):= (\tilde I^k_h f^{e})(a^{-1}(p)), p \in \mathcal M$,
where $f^{e}:U \rightarrow \mathbb R$ is the usual extension of $f$ given by $f^{e}(p):=f(a(p))$ and $\tilde I^k_h: C^0(\mathcal M_h) \rightarrow \tilde S^k_h$ is the Lagrange interpolation operator. 
The estimates \eqref{eq:interpol} and \eqref{eq:inverse} can then be proved using the techniques developed in \cite[Section 2]{Dem09}. \\
Let us next turn to Assumption {\bf (A3)} where we restrict ourselves to the case $k=2$. To do so, fix $x_h \in (S^2_h)^3$ with $\dgh \geq c_0>0, | \nabla_{\mathcal M} x_h | \leq c_1$ and define 
$\hat \nu_h:= \tilde I^2_h \bigl( \frac{ \tilde \nu_h}{| \tilde \nu_h|} \bigr) \circ a^{-1}$, where $\tilde \nu_h \in (\tilde S^2_h)^3$ is the unique solution of 
\begin{equation} \label{eq:deftnuh}
\sum_{\tilde T} | \tilde T| \sum_{\tilde p \in \mathcal N_2(\tilde T)} \tilde \nu_h(\tilde p) \cdot \tilde \varphi_h (\tilde p) = \sum_{\tilde T} |  \tilde T |  \sum_{\tilde p \in \mathcal N_2(\tilde T)}  \nu_h^e(\tilde p) \cdot \tilde \varphi_h(\tilde p) \qquad \forall \tilde \varphi_h \in (\tilde S^2_h)^3.
\end{equation}
Here, $\nu_h^e$ is the extension of $\nu_h$, recall \eqref{nueq}, and we denote for a triangle  $\tilde T \subset \mathcal M_h$ 
\begin{displaymath}
\mathcal N_2(\tilde T):= \lbrace \tilde p \in \tilde T \, | \, \tilde p \mbox{ is either a vertex or the midpoint of an edge of } \tilde T \rbrace.
\end{displaymath}
Clearly, $\hat \nu_h \in (S^2_h)^3$. In order to verify  \eqref{eq:tnuh} we  let $x \in W^{3,\infty}(\mathcal M)^3$ with $\dg \geq c_0$.  Using $\tilde \varphi_h = \tilde \nu_h - \tilde I^2_h \nu^e$ in \eqref{eq:deftnuh} we obtain
\begin{eqnarray*}
\lefteqn{ \hspace{-2cm} 
 \sum_{\tilde T} | \tilde T | \sum_{\tilde p \in \mathcal N_2(\tilde T)} | \tilde \nu_h(\tilde p) - \tilde I^2_h \nu^e(\tilde p) |^2 =  \sum_{\tilde T} | \tilde T | \sum_{\tilde p \in \mathcal N_2(\tilde T)}  \bigl( \tilde \nu_h(\tilde p) - \tilde I^2_h \nu^e(\tilde p) \bigr) \cdot \bigl(
 \nu_h^e(\tilde p) - \tilde I^2_h \nu^e(\tilde p) \bigr) } \\
 & = & \sum_{\tilde T} | \tilde T | \sum_{\tilde p \in \mathcal N_2(\tilde T)}  \bigl( \tilde \nu_h(\tilde p) - \tilde I^2_h \nu^e(\tilde p) \bigr) \cdot \bigl(
 \nu_h^e(\tilde p) - \nu^e(\tilde p) \bigr) 
 \end{eqnarray*}
 and hence 
 \begin{equation} \label{eq:a1}
  \sum_{\tilde T} | \tilde T | \sum_{\tilde p \in \mathcal N_2(\tilde T)} | \tilde \nu_h(\tilde p) -  \nu^e(\tilde p) |^2 \leq  \sum_{\tilde T} | \tilde T | \sum_{\tilde p \in \mathcal N_2(\tilde T)} |  \nu^e_h(\tilde p) - \nu^e(\tilde p) |^2.
\end{equation}
Let us fix $\tilde T \subset \mathcal M_h$ and $\tilde p \in \mathcal N_2(\tilde T)$. Setting $p=a(\tilde p)$ and $T=a(\tilde T)$ we have
 \begin{displaymath}
 | \nu^e_h(\tilde p) - \nu^e(\tilde p) |
=  | [\frac{1}{\sqrt{\dgh}} \mu_j \underline D_{j+1} x_h \wedge \underline D_{j+2} x_h - \frac{1}{\sqrt{\dg}} \mu_j \underline D_{j+1} x \wedge \underline D_{j+2} x](p) | \leq c | \nabla_{\mathcal M} (x_h -x)(p)|,
\end{displaymath}
where $c$ depends on $c_0,c_1$ and $\Vert \nabla_{\mathcal M} x \Vert_{L^\infty}$. Inverse and interpolation inequalities then imply that
\begin{equation} \label{eq:a2}
 | \nu^e_h(\tilde p) - \nu^e(\tilde p) | 
\leq c h^{-1} \bigl( \int_T | \nabla_{\mathcal M}(x_h - x) |^2 do \bigr)^{\frac{1}{2}}+ c h^2 \Vert x \Vert_{W^{3,\infty}}.
\end{equation} 
Using the equivalence of norms on finite dimensional spaces,  a transformation argument as well as \eqref{eq:a1} and \eqref{eq:a2} we see that
\begin{eqnarray*}
\int_{\mathcal M} |  \hat \nu_h - I^2_h \nu |^2 \, do   & \leq &  c  \int_{\mathcal M_h} |  \tilde I^2_h \bigl( \frac{\tilde \nu_h}{| \tilde \nu_h |} \bigr)  - \tilde I^2_h \nu^e |^2 do \leq
c \sum_{\tilde T} | \tilde T | \sum_{\tilde p \in \mathcal N_2(\tilde T)} | \frac{\tilde \nu_h(\tilde p)}{| \tilde \nu_h(\tilde p)| } - \nu^e(\tilde p) |^2  \\
 & \leq & c   \sum_{\tilde T} | \tilde T | \sum_{\tilde p \in \mathcal N_2(\tilde T)} | \tilde \nu_h(\tilde p)   - \nu^e(\tilde p) |^2   \leq   c \bigl( \Vert \nabla_{\mathcal M} (x_h - x) \Vert^2 +  h^4 \bigr),
\end{eqnarray*}
where $c$ only depends on $c_0, c_1, \Vert x \Vert_{W^{3,\infty}}$ and constants from inverse and interpolation inequalities.  The bound \eqref{eq:tnuh} now follows with the help of
an interpolation estimate. \\[2mm]
\subsubsection{Fully discrete numerical scheme}
\noindent
In order to derive a fully discrete approximation of \eqref{eq:weakdisc} we choose a time step  $\tau>0$ and set $t^m=m \tau, ~m=0,\ldots,M$, where $M=\frac{T}{\tau}$. 
Our fully discrete approximation of \eqref{eq:weakdisc} now reads: let $x_h^0:=I^k_h x_0$ then for $m=0,1,\ldots,M-1$ find $x_h^{m+1}\in (S^k_h)^3$ such that 
\begin{eqnarray} 
\lefteqn{ \hspace{-2cm}
\frac1{\tau}\int_{\mathcal M} A(\nu_h^m,\sqrt{\dghm}) (x_{h}^{m+1}-x_h^{m}) \cdot \chi_h \, do + \int_{\mathcal M} (G_h^m)^{-1} \nabla_{\mathcal M}  x_h^{m+1} : \nabla_{\mathcal M} \chi_h \sqrt{\dghm} \, do }
\nonumber  \\
& &  + \frac{1+\alpha}{\alpha}  \int_{\mathcal M}  \nabla_{\mathcal M} x_h^{m+1} : \nabla_{\mathcal M}  \bigl( P_h^m \chi_h \bigr) \, do =0  \qquad \forall \chi_h \in (S^k_h)^3. \label{eq:weakfulldisc}
\end{eqnarray}
Here, 
$A(w,\rho) \in \mathbb R^{3 \times 3}$ is given by \eqref{eq:defH}, while 
$$
G_{h,ij}^m =   \underline D_i x_h^m \cdot \underline D_j x_h^m + \mu_i \mu_j, \quad i,j=1,2,3, \quad 
P_h^m :=I_3-\hat  \nu^m_h \otimes \hat \nu^m_h.
$$


\subsubsection{Fully discrete second order in time numerical scheme}
By following the authors in \cite{DN26} we introduce a second order in time finite element approximation of \eqref{eq:weakdisc} that takes the form: let $x_h^0:=I^k_h x_0$ then for $m=0,1,\ldots,M-1$ first find $x_h^{m+\frac12}\in (S^k_h)^3$ such that \begin{eqnarray} 
\lefteqn{ \hspace{-2cm}
\frac1{\frac12\tau}\int_{\mathcal M} A(\nu_h^m,\sqrt{\dghm}) (x_{h}^{m+\frac12}-x_h^{m}) \cdot \chi_h \, do + \int_{\mathcal M} (G_h^m)^{-1} \nabla_{\mathcal M}  x_h^{m+\frac12} : \nabla_{\mathcal M} \chi_h \sqrt{\dghm} \, do }
\nonumber  \\
& &  +~ \frac{1+\alpha}{\alpha}  \int_{\mathcal M}  \nabla_{\mathcal M} x_h^{m+\frac12} : \nabla_{\mathcal M}  \bigl( P_h^m \chi_h \bigr) \, do =0  \qquad \forall \chi_h \in (S^k_h)^3 \label{eq:weakfulldisc1}
\end{eqnarray}
and then find $x_h^{m+1}\in (S^k_h)^3$ such that
\begin{eqnarray} 
\lefteqn{ \hspace{-2cm}
\frac1{\tau}\int_{\mathcal M} A(\nu_h^{m+\frac12},\sqrt{\dghmph}) (x_{h}^{m+1}-x_h^{m}) \cdot \chi_h \, do }\nonumber\\
&&+~ \frac12\int_{\mathcal M} (G_h^{m+\frac12})^{-1} \nabla_{\mathcal M}  (x_h^{m+1}+x_h^m) :  \nabla_{\mathcal M} \chi_h \sqrt{\dghmph} \, do 
\nonumber  \\
& &  + ~\frac{1+\alpha}{2\alpha}  \int_{\mathcal M}  \nabla_{\mathcal M} (x_h^{m+1} +x_h^m): \nabla_{\mathcal M}  \bigl( P_h^{m+\frac12} \chi_h \bigr) \, do =0  \qquad \forall \chi_h \in (S^k_h)^3. \label{eq:weakfulldisc2}
\end{eqnarray}

\subsection{Numerical Computations}\label{sec5.2}
We implemented \eqref{eq:weakfulldisc} and \eqref{eq:weakfulldisc1},  \eqref{eq:weakfulldisc2} within the Finite Element Toolbox ALBERTA, see \cite{alberta}. 
To do so with the above choice of $S^k_h$ one needs to transform  integrals over $\mathcal M=a(\mathcal M_h)$ to integrals over $\mathcal M_h$, which yields
a variational form that uses  $(\tilde S^k_h)^3$ as test and solution space.
In all the computations we set $k=2$ and took the discrete hypersurface $\mathcal{M}_h$ to be a piecewise linear approximation of the unit sphere. We denote by 
$\tilde{\mathcal T}_{\mathcal M_h} $ a triangulation of $\mathcal{M}_h$ and define $\hmaxm:=\max_{\tilde T \in \tm} \mbox{diam}(\tilde T)$. Associated with the triangulation $\tm$ of $\mathcal{M}_h$, at each time $t^m$, we have the decomposition 
$\tgam:=\lbrace x^m_h(a(\tilde T)),  \tilde T  \in \tm \rbrace = \lbrace \tilde x^m_h(\tilde T),  \tilde T  \in \tm \rbrace$
of the discrete surface $\Gamma^m_h$. 



\subsubsection{Experimental order of convergence}\label{s:eoc}
In order to investigate the experimental order of convergence (eoc) of our scheme we consider the evolution of a two-dimensional sphere with initial radius $R_0$. It is not difficult to verify that $x:\mathcal M \times [0,\frac{R_0^2}{4}) \rightarrow \mathbb R^3, x(p,t)=\sqrt{R_0^2-4t} \,  p$ 
is a solution of \eqref{eq:repar}. 
\noindent
When investigating the experimental order of convergence we monitor the following relative errors
$$
\mathcal{E}_1:= \max_{0 \leq m \leq M} \frac{ \| x_h^m-x(t^m)\|^2}{ \|  x(t^m)\|^2}, \quad 
\mathcal{E}_2:=\max_{0 \leq m \leq M}  \frac{ \|\nabla_{\mathcal M} (x_h^m-x(t^m))\|^2}{ \|\nabla_{\mathcal M} x(t^m)\|^2}, 
$$
$$
\mathcal{E}_3:=\max_{0 \leq m \leq M} \frac{ \| x_h^m-x(t^m)\|^2_{H^1}}{ \| x(t^m)\|^2_{H^1}}.
$$
\noindent
For the numerical simulations in this section we set $\alpha=0.1$, $R_0=2$, $T=0.6$ and implemented the second order scheme \eqref{eq:weakfulldisc1}, \eqref{eq:weakfulldisc2} with $\tau = 0.0001$. We used the second order scheme for the computations because the first order scheme produced errors, $\mathcal{E}_i$, $i=1,2,3$, that were dominated by time discretisation errors and prohibitively small values for the time step $\tau$ were required to eliminate this effect, which was not the case for the second order scheme.  In Section \ref{s:fs}, we compare the computational efficiency of the first and second order schemes. \\
\noindent
In Table \ref{t:eoc} we display values of $\mathcal{E}_i$, $i=1,2,3$, to 3 decimal places, evaluated using a quadrature rule of degree $8$.
From this table we see the expected order of convergence with eoc's close to $4$ for  $\mathcal{E}_i$, $i=2,3$, which align with the bounds obtained in Theorem \ref{thm:main}, while for $\mathcal{E}_1$ we see eoc's close to $6$.
\begin{center}
\begin{table}[!h]
\centering
 \begin{tabular}{ |c|c|c|c|c|c|c| }
 \hline
$\hmaxm$ & $\mathcal{E}_1$ & $eoc_1$ & $\mathcal{E}_2$ & $eoc_2$ & $\mathcal{E}_3$ &$eoc_3$ \\\hline 
$3.298\times 10^{-1}$  & $2.550\times 10^{-7}$ & - &         $8.179\times 10^{-5}$ & - &         $5.443\times 10^{-5}$ & -     \\ 
$1.630\times 10^{-1}$& $2.232\times 10^{-9}$ & 6.718 & $3.195\times 10^{-6}$ & 4.607 & $2.129\times 10^{-6}$ & 4.611  \\ 
$8.407\times 10^{-2}$ & $2.640\times 10^{-11}$ & 6.703& $1.630\times 10^{-7}$ & 4.496 &  $1.087\times 10^{-7}$ & 4.497   \\ 
$4.439\times 10^{-2}$ & $3.769\times 10^{-13}$ & 6.653 & $9.669\times 10^{-9}$ & 4.424 & $6.446\times 10^{-9}$ &4.424\\  
$2.230\times 10^{-2}$& $5.770\times 10^{-15}$ & 6.072 & $6.107\times 10^{-10}$ & 4.034 & $4.012\times 10^{-10}$ &4.034\\  \hline
\end{tabular}
\caption{Spherical simulation with $\alpha=0.1$:
$\mathcal{E}_i$ and $eoc_i$, $i=1,2,3$, over the time interval $[0,0.6]$, with $\tau =0.0001$ using the second order scheme \eqref{eq:weakfulldisc1}, \eqref{eq:weakfulldisc2}.} 
\label{t:eoc}
\end{table}
\end{center}

\subsubsection{Effect of $\alpha$ on mesh quality - singular mean curvature flow}
We now  follow  the authors in \cite{EF17,KLL19} in considering the evolution of 
a dumbbell shaped surface given by the distance function
\begin{equation}
d(x)=x_1^2+x_2^2+F(x_3^2)-0.04
\label{db}
\end{equation}
with $F(s)=2s(s-199/200)$.
In all the numerical simulations in this section, Figures \ref{initial_db}--\ref{alpha_mesh}, we implemented the first order scheme \eqref{eq:weakfulldisc} with $h_{max} \approx 0.0084$ and $\tau = 0.0001$. \\
\noindent
In Figure \ref{initial_db} we show the discrete reference hypersurface $\mathcal M_h$ (left image) together with the 
triangulation ${\mathcal T}^0_h $ given by the elements $x^0_h(a(\tilde T))$, $\tilde T \in \tm$ (right image). 
The mean curvature flow for the surface is shown in Figures \ref{evol_db} and \ref{evol_db2}, with Figure \ref{evol_db} displaying the surface at times $t=0.01,0.03,0.05$ and Figure \ref{evol_db2} displaying an enlarged plot of the surface at times $t=0.04$ and $t=0.06$. In each figure we present results for two values of $\alpha$, $\alpha=1.0$ (left images) and $\alpha=0.04$ (right images). In Figure \ref{evol_db2} we see the surface develops a neck pinch singularity in finite time. Similar results to Figures \ref{evol_db} and \ref{evol_db2} are shown in Figure 14 in \cite{EF17}  and Figure 13.3  in \cite{KLL19}.  \\
\noindent
From Figures \ref{evol_db} and \ref{evol_db2} we see that the elements in the triangulations $\tgam$ appear to be of a higher quality for $\alpha=0.04$ than for $\alpha=1.0$. To investigate this further we follow the authors in \cite{EF17} and evaluate the quantity 
$$
\sigma^m_{max} = \max_{T \in \tgam}\frac{\mbox{diam}(T)}{r(T)}
$$
where $r(T)$ is the radius of the largest ball contained in $T$, such that small values of $\sigma_{max}^m$ imply that there are no simplices in $T$ with sharp angles.  
In Figure \ref{alpha_mesh} we show the effect that $\alpha$ has on (i) $\sigma_{max}^m$ and (ii) the surface area of $\tgam$.  In the left image we plot $\sigma_{max}^m$ against time for $\alpha$ ranging from $0.01$ to $1.0$. 
From this image we see that the mesh quality of the triangulation $\tgam$ increases as $\alpha$ decreases. 
In the middle image in Figure \ref{alpha_mesh} we show the decrease of the discrete surface area of $\tgam$ under the mean curvature flow. In this image we plot the surface area of $\tgam$ against time for $\alpha$ ranging from $0.01$ to $1.0$. The right image displays an enlarged section of the centre image. From these two images we see that the discrete surface area decreases at effectively the same rate for $\alpha=1.0, 0.1, 0.04$, while for smaller values of $\alpha$, namely $\alpha=0.02$ and $\alpha=0.01$, towards the end of the simulation, specifically for $t>0.045$, the decrease in the discrete surface area occurs at a slightly faster rate. Similar behaviour is observed in Figure 16  in \cite{EF17}. 


\begin{figure}[h]
\centering
\subfigure{\includegraphics[width = 0.32\textwidth,angle=90]{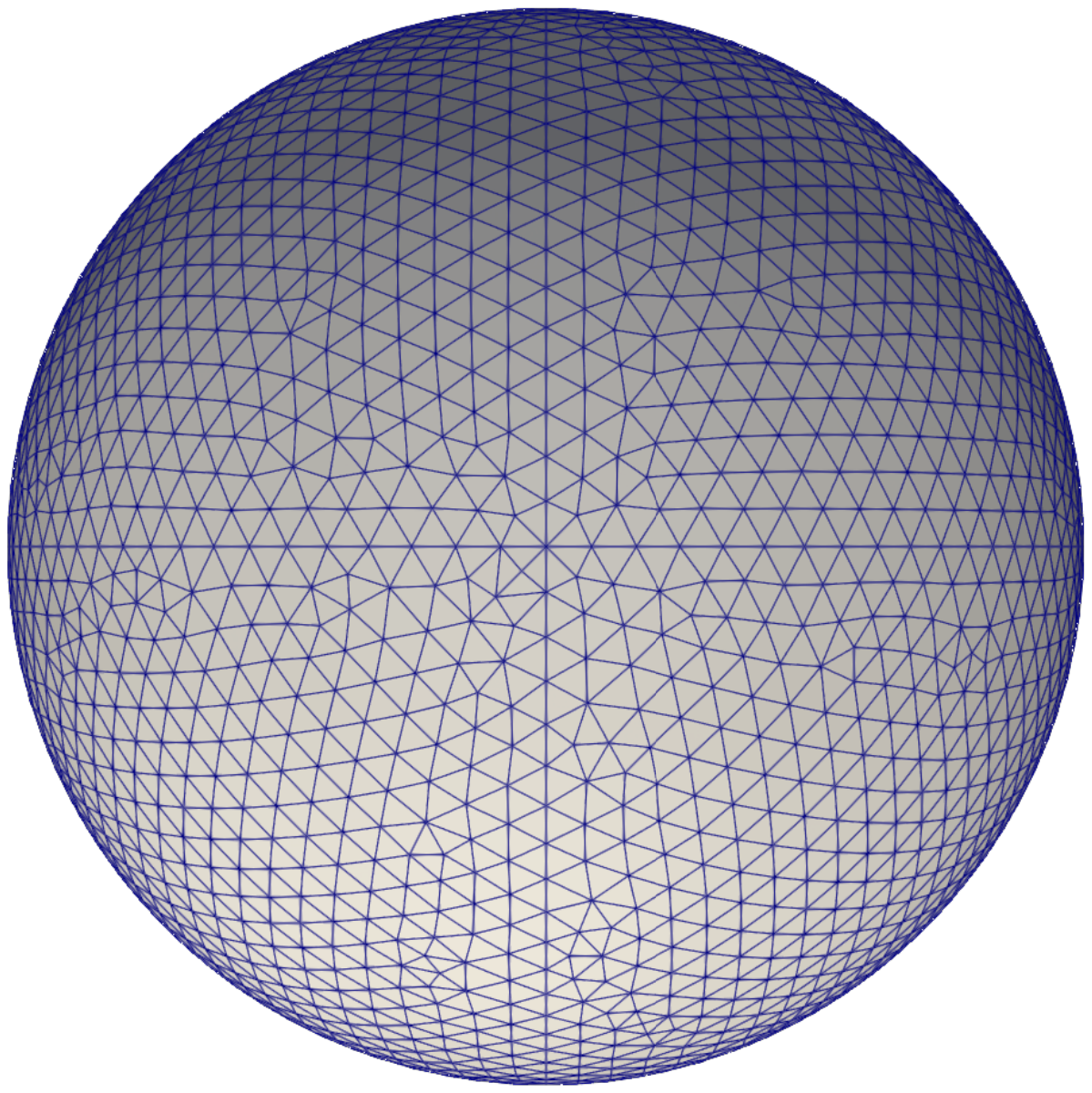}} 
\subfigure{\includegraphics[width = 0.32\textwidth,angle=90]{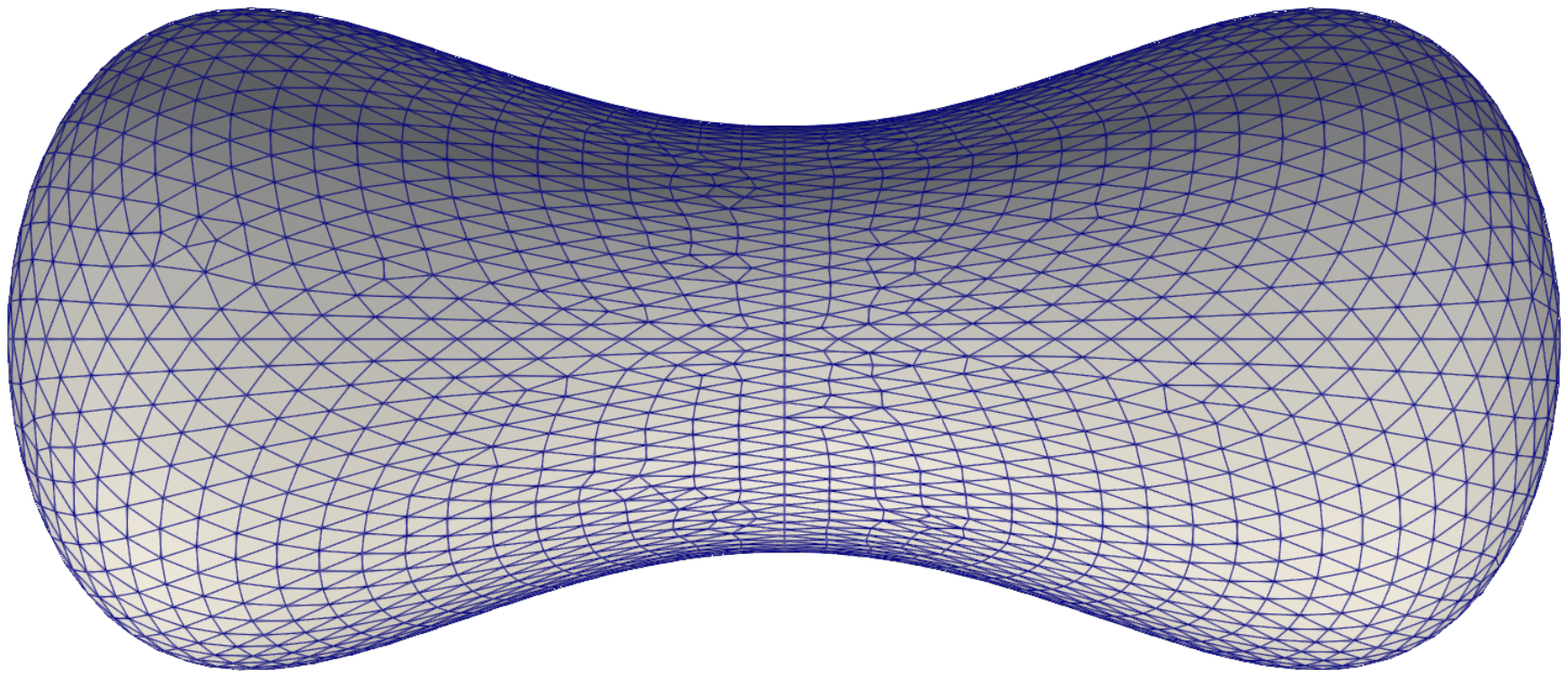}} 
\caption{Reference hypersurface $\mathcal M_h$ (left),  initial dumbbell shaped triangulation $\mathcal T^0_h$ (right).} 
\label{initial_db}
\end{figure}

\begin{figure}[h]
\centering
\subfigure{\includegraphics[width = 0.32\textwidth,angle=90]{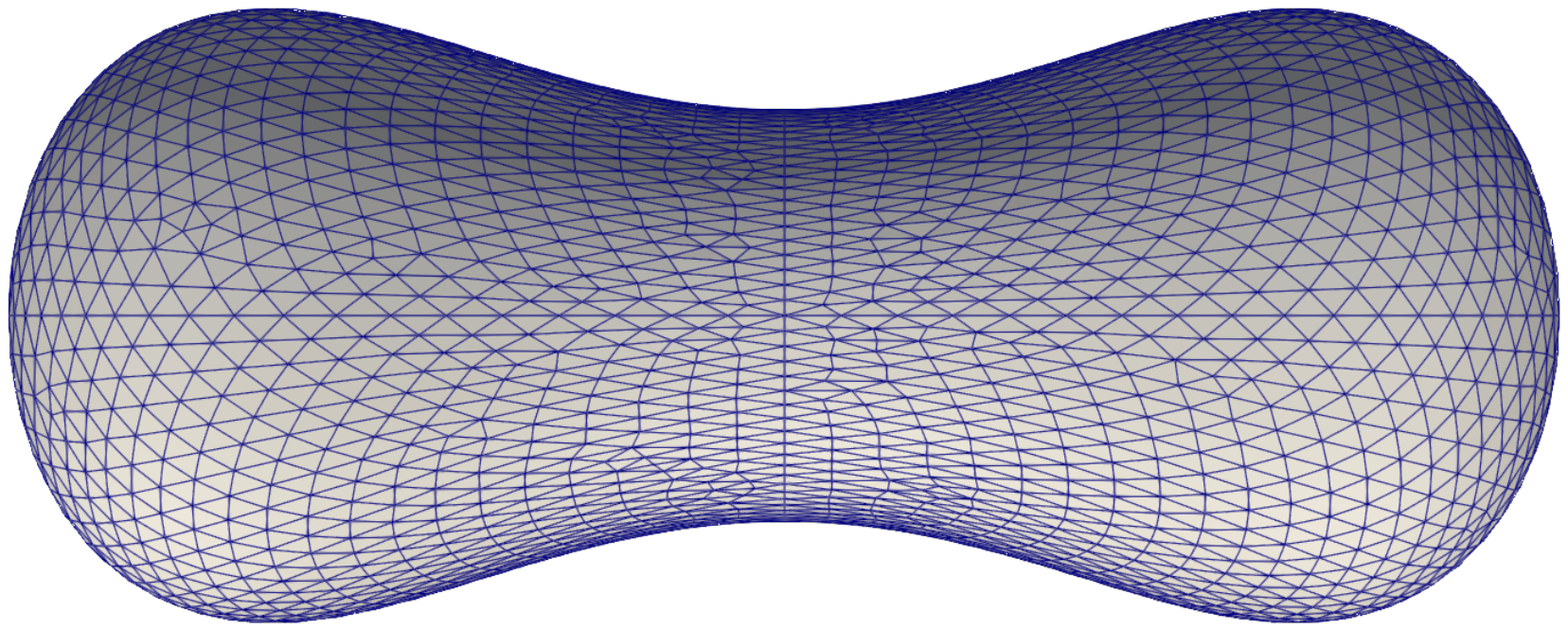}}  \quad
\subfigure{\includegraphics[width = 0.32\textwidth,angle=90]{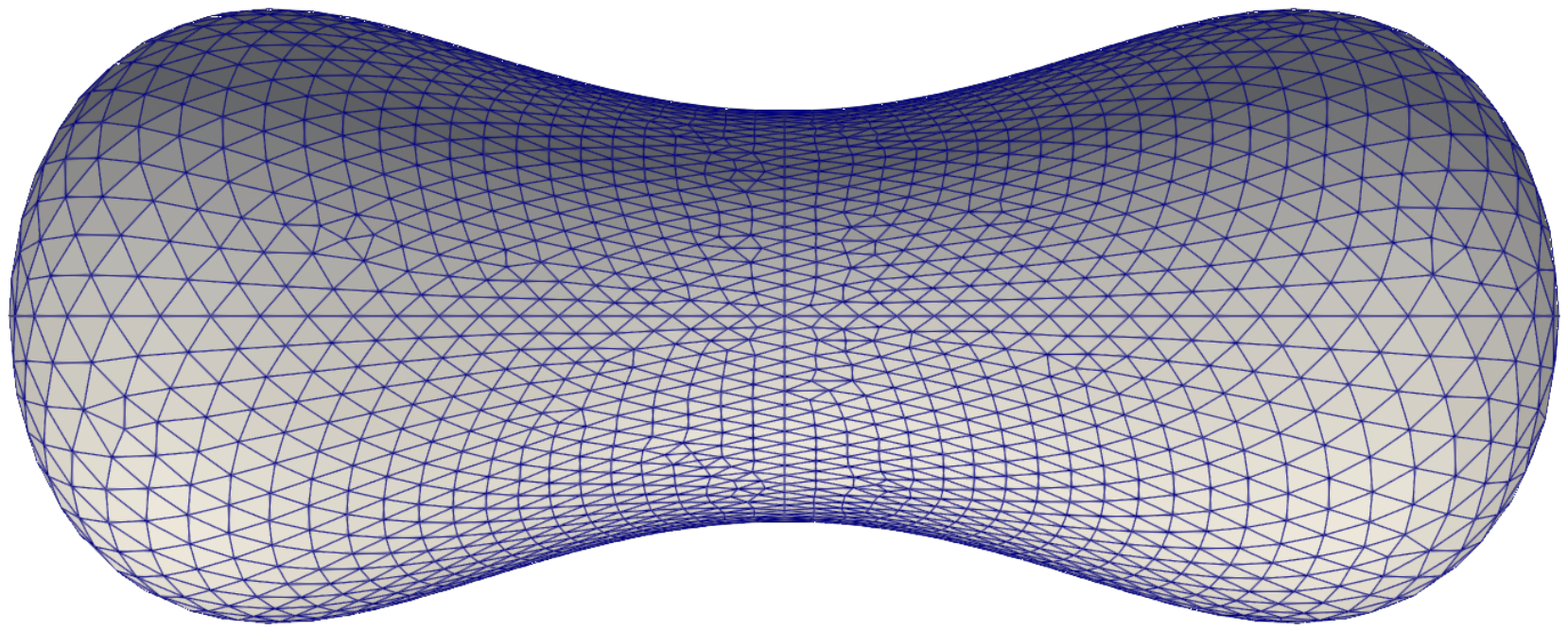}} \\[-12mm]
\subfigure{\includegraphics[width = 0.32\textwidth,angle=90]{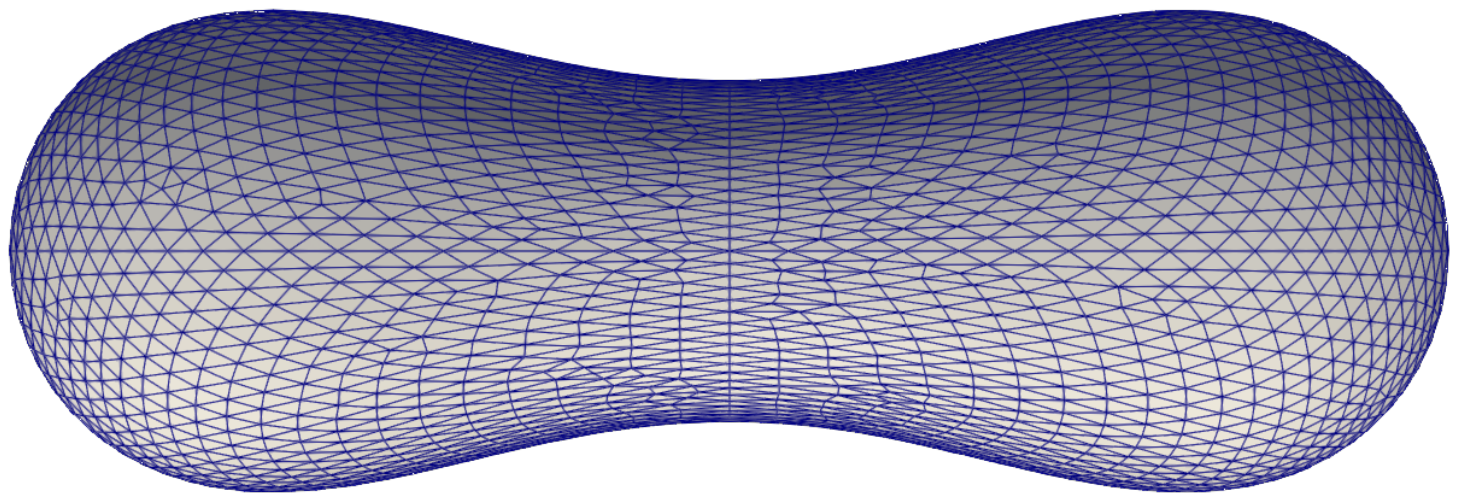}}  \quad
\subfigure{\includegraphics[width = 0.32\textwidth,angle=90]{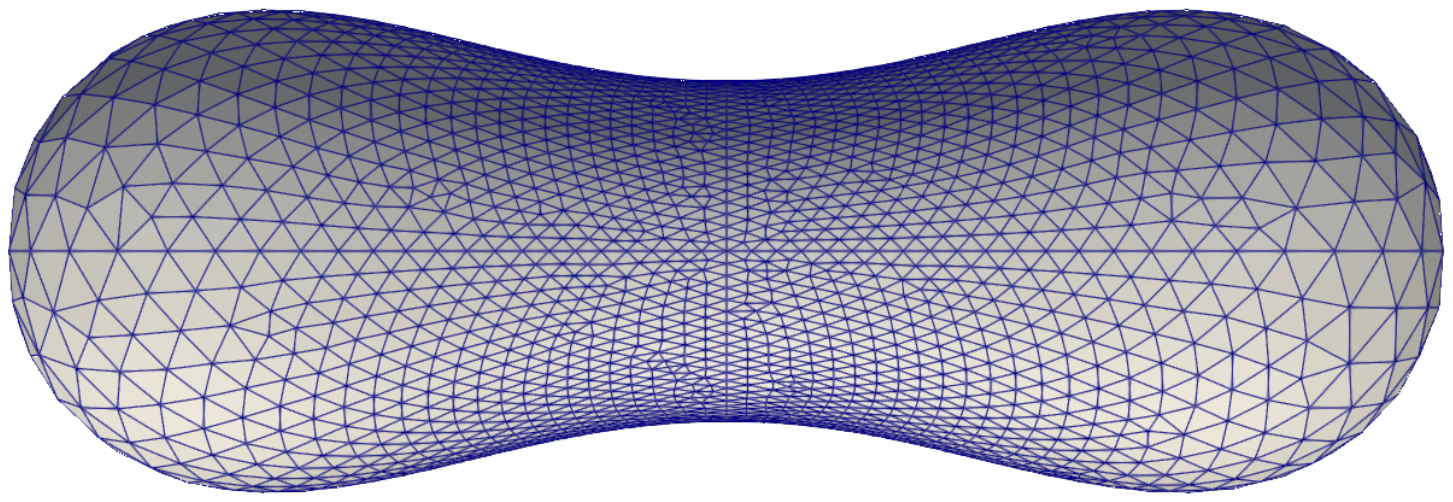}} \\[-18mm]
\subfigure{\includegraphics[width = 0.32\textwidth,angle=90]{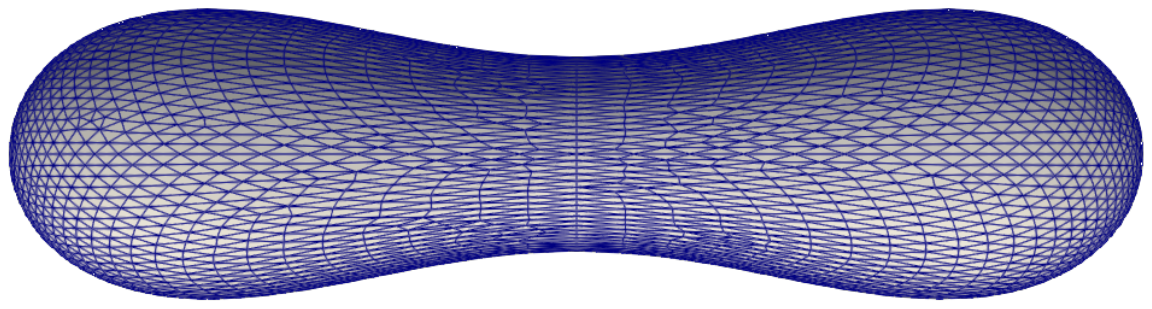}}  \quad
\subfigure{\includegraphics[width = 0.32\textwidth,angle=90]{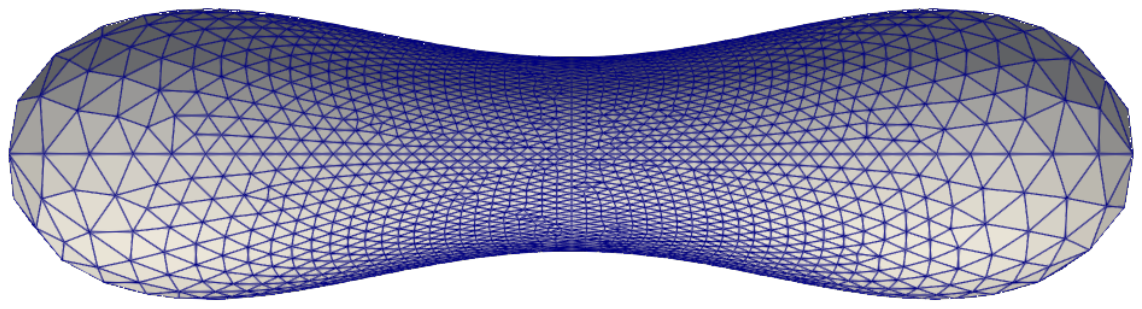}} \\[-10mm]
\caption{Dumbbell simulation: the triangulation $\tgam$ at $t=0.01$ (first row), $t=0.03$ (second row) and $t=0.05$ (third row) for $\alpha=1.0$ (left) $\alpha = 0.04$ (right).} 
\label{evol_db}
\end{figure}

\begin{figure}[h]
\centering
\subfigure{\includegraphics[width = 0.32\textwidth,angle=90]{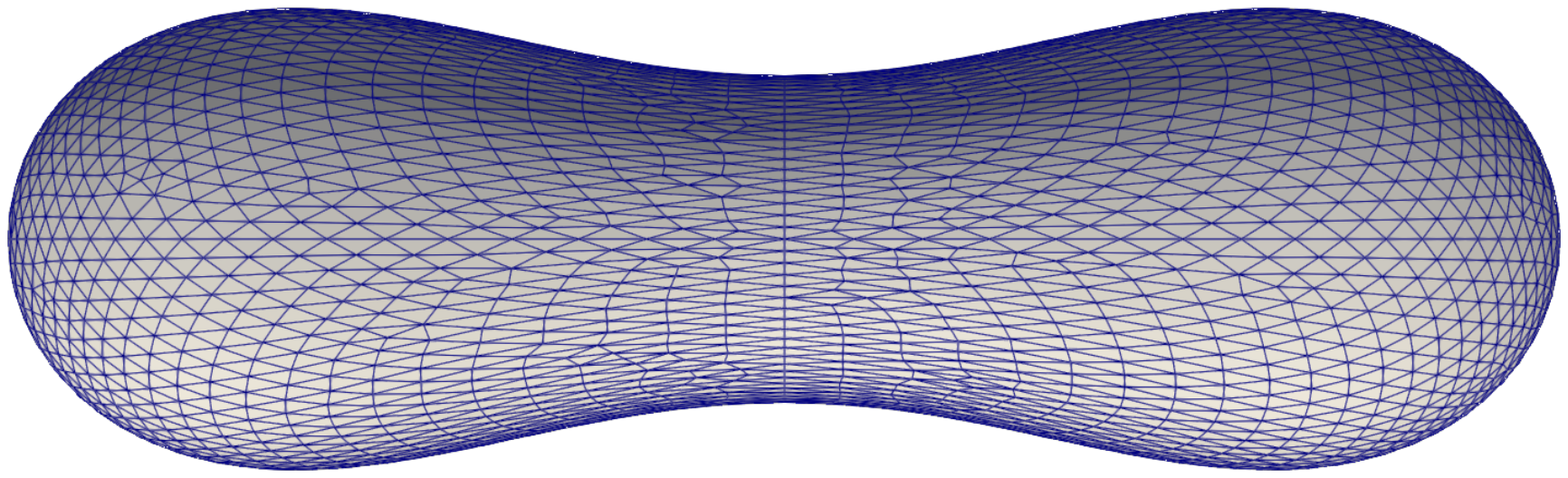}} \quad
\subfigure{\includegraphics[width = 0.32\textwidth,angle=90]{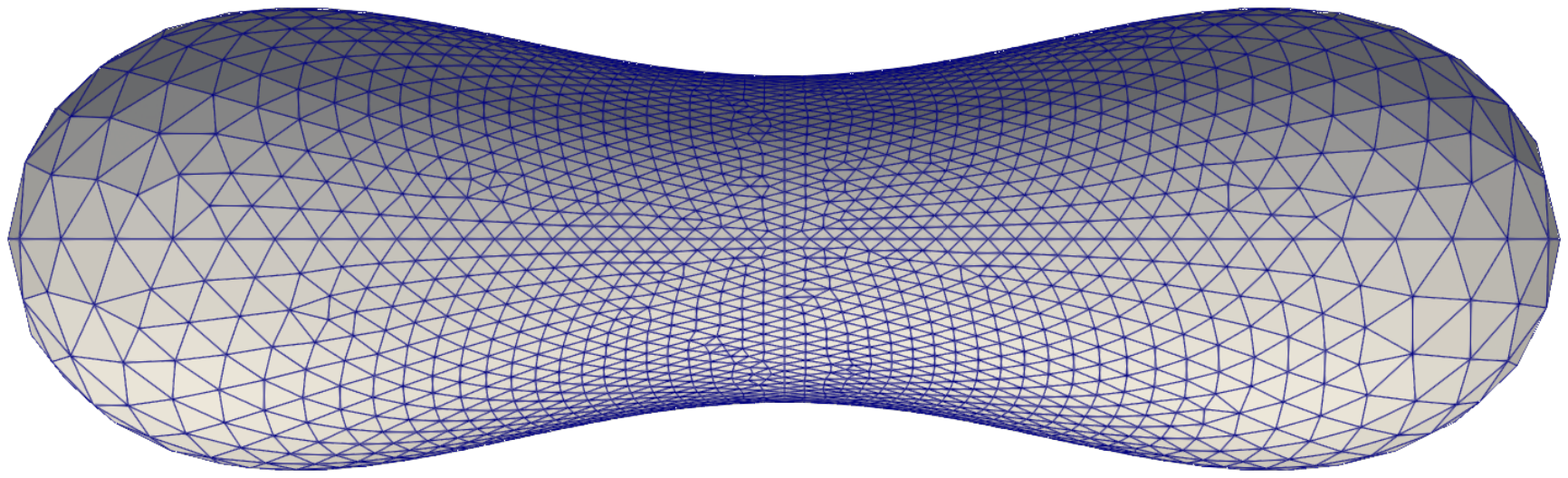}} \\[-18mm]
\subfigure{\includegraphics[width = 0.32\textwidth,angle=90]{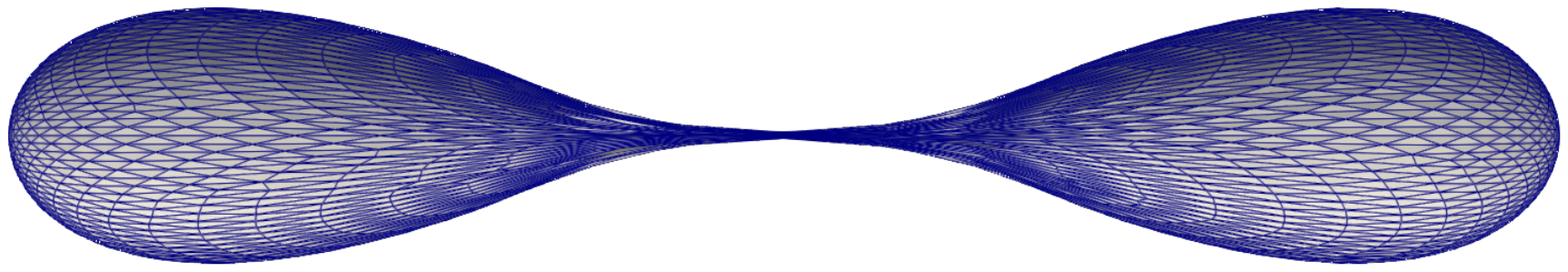}}  \quad
\subfigure{\includegraphics[width = 0.32\textwidth,angle=90]{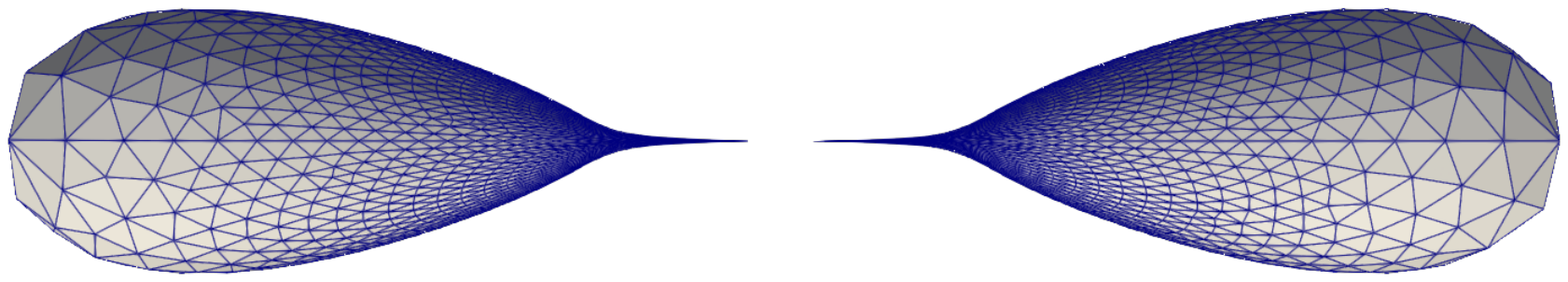}} \\[-10mm]
\caption{Dumbbell simulation: the triangulations $\tgam$ at $t^m=0.04$ (first row) and $t=0.06$ (second row) for $\alpha=1.0$ (left) $\alpha = 0.04$ (right).} 
\label{evol_db2}
\end{figure}

\begin{figure}[h]
\centering
\subfigure{\includegraphics[width = 0.32\textwidth]{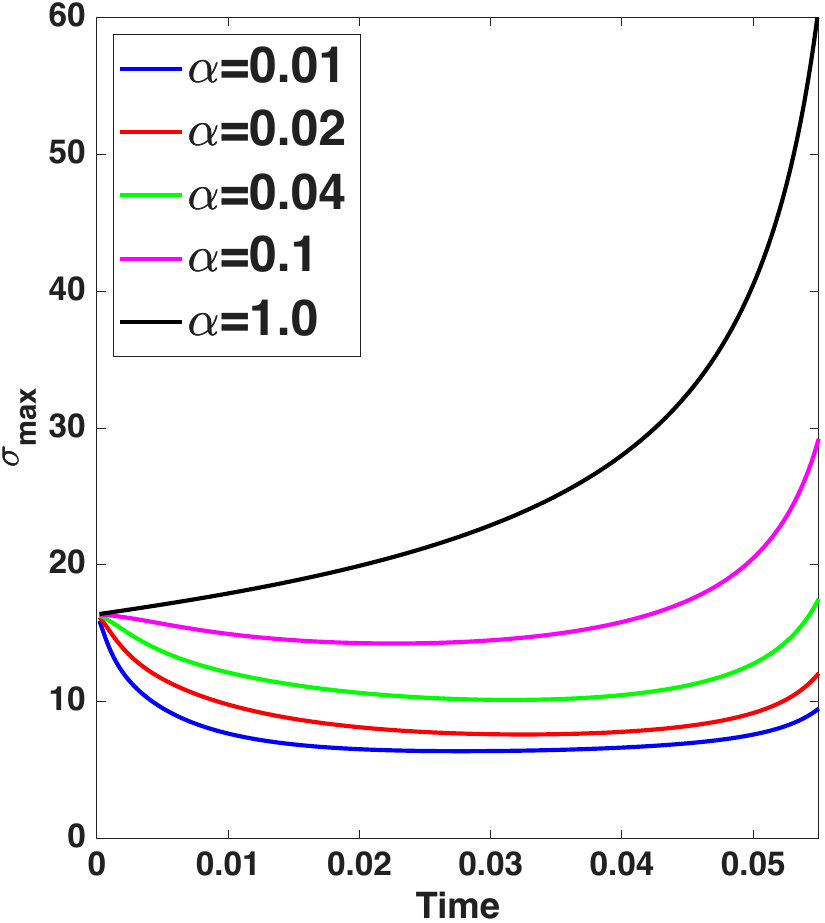}} 
\subfigure{\includegraphics[width = 0.32\textwidth]{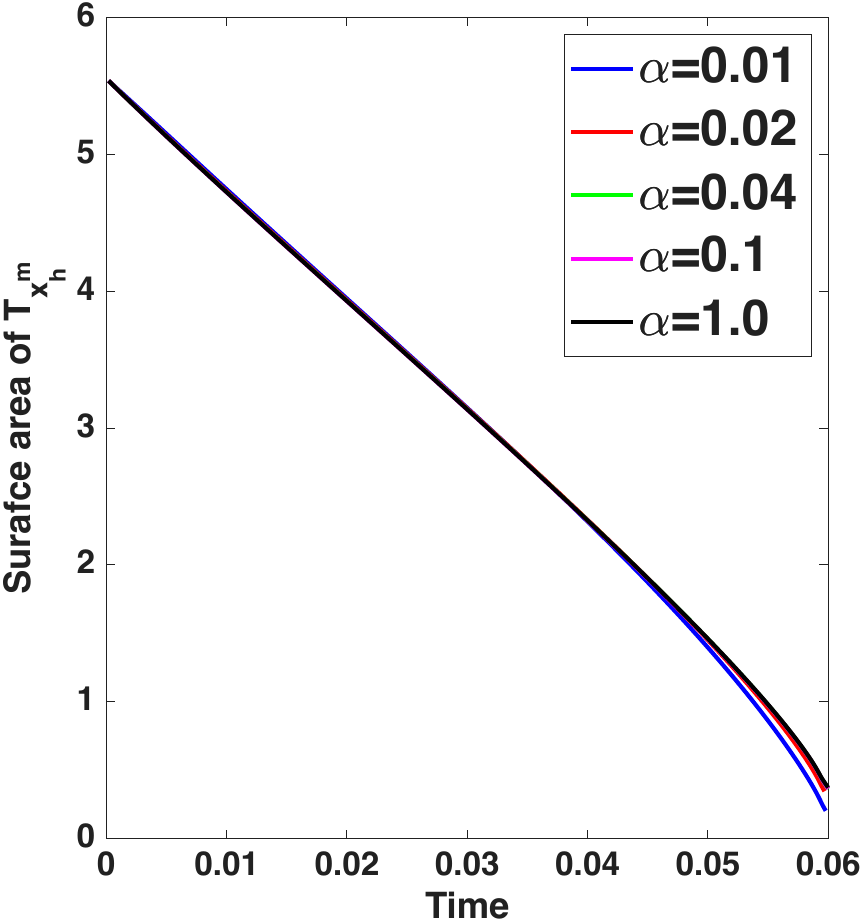}} 
\subfigure{\includegraphics[width = 0.32\textwidth]{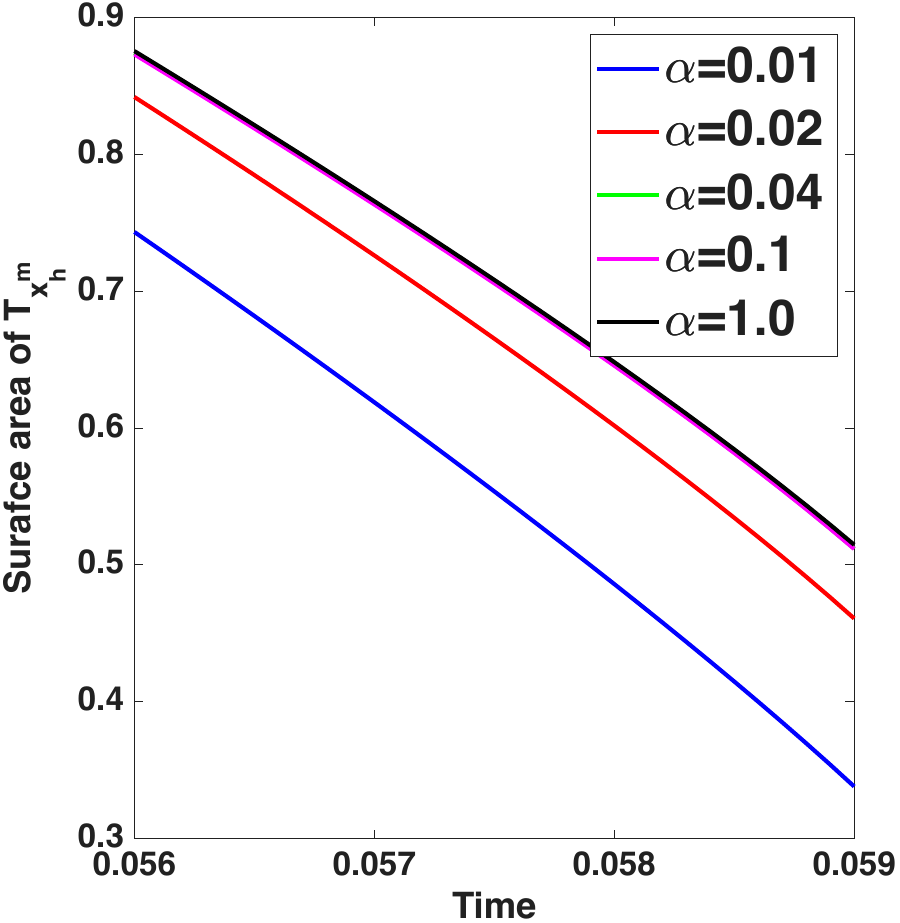}} 
\caption{Dumbbell simulation. The effect of $\alpha$ on (i) $\sigma_{max}$ (left) and (ii) the surface area of $\tgam$ (centre and right).} 
\label{alpha_mesh}
\end{figure}

%

\subsubsection{Computational efficiency of the second order scheme \eqref{eq:weakfulldisc1}, \eqref{eq:weakfulldisc2}} \label{s:fs}
In Table \ref{first-second} we compare the second order scheme  \eqref{eq:weakfulldisc1}, \eqref{eq:weakfulldisc2} with the first order scheme  \eqref{eq:weakfulldisc}. As in Section \ref{s:eoc} we consider the evolution of a two-dimensional sphere with initial radius $R_0=2$, however for these computations we fixed the spatial discretisation to be a triangulation $\tm$ of $\mathcal{M}_h$ with $\hmaxm=4.439\times 10^{-2}$ and varied the time step from $\tau=0.05$ to $\tau=0.0005$. In this simulation we set $\alpha=0.1$. From this table we see that already for $\tau = 0.001$ the spatial error dominates for the second order scheme \eqref{eq:weakfulldisc1}, \eqref{eq:weakfulldisc2} and to reach approximately the same error as the second order scheme with $\tau =0.05$ the first order scheme \eqref{eq:weakfulldisc} needs to employ a time step size of $\tau=0.0005$, which results in a CPU time that is over $14$ times larger.

\begin{center}
\begin{table}[!h]
\centering
 \begin{tabular}{ |c|c|c|c|c| }
 \hline
~&\multicolumn{2}{c|}{\eqref{eq:weakfulldisc}} & \multicolumn{2}{c|}{\eqref{eq:weakfulldisc1}, \eqref{eq:weakfulldisc2}}\\
$\tau$ &$\displaystyle{\max_{0 \leq m \leq M} \frac{\| x_h^m-x(t^m)\|^2}{\|x(t^m)\|^2}}$& CPU time [s] & $\displaystyle{\max_{0 \leq m \leq M}\frac{ \|x_h^m-x(t^m)\|^2}{\|x(t^m)\|^2}}$ & CPU time [s]  \\[3mm] \hline \rule{0pt}{2.5ex} 
$0.05$ & $1.128\times 10^{-5}$  &  $38$&  $3.720\times 10^{-9}$ & $69$\\ 
$0.01$ & $9.250\times 10^{-7}$  &  $144$&  $1.300\times 10^{-11}$ & $238$\\ 
$0.005$ & $2.516\times 10^{-7}$ &  $236$  & $1.204\times 10^{-12}$ & $379 $  \\ 
$0.001$ & $1.076\times 10^{-8}$ &  $743$&  $3.844\times 10^{-13}$&  $1122$ \\ 
$0.0005$ & $2.711\times 10^{-9}$ &  $1108$&  $3.808\times 10^{-13}$&$2083$ \\ \hline
\end{tabular}
\caption{Spherical simulation with $\alpha=0.1$. Errors and CPU times over the time interval $[0,0.2]$, with $\hmaxm=4.439\times 10^{-2}$.} 
\label{first-second}
\end{table}
\end{center}

\subsubsection{Simplified fully discrete numerical scheme}
 Let us briefly introduce a simplified version of the scheme \eqref{eq:weakfulldisc},   in which the integrals over $\mathcal M$ are replaced with integrals over $\mathcal M_h$ leading to the following form:
let $\tx_h^0:=\tilde I^k_h x^e_0$ then for $m=0,1,\ldots,M-1$ find $\tx_h^{m+1}\in (\tilde S^k_h)^3$ such that 
\begin{eqnarray} 
\lefteqn{ \hspace{-2cm}
\frac1{\tau}\int_{\mathcal M_h} A(\nuu_h^m,\sqrt{\dghm}) (\tx_{h}^{m+1}-\tx_h^{m}) \cdot \tchi_h \, do + \int_{\mathcal M_h} (G_h^m)^{-1} \nabla_{\mathcal M_h}  \tx_h^{m+1} : \nabla_{\mathcal M} \tchi_h \sqrt{\dghm} \, do }
\nonumber  \\
& &  + \frac{1+\alpha}{\alpha}  \int_{\mathcal M_h}  \nabla_{\mathcal M_h} \tx_h^{m+1} : \nabla_{\mathcal M_h}  \bigl( P_h^m \tchi_h \bigr) \, do =0  \qquad \forall \tchi_h \in (\tilde S^k_h)^3. \label{eq:weakfulldisc_simple}
\end{eqnarray}
Here, $G_h^m$ is given by 
$$
G_{h,ij}^m =   \underline D_i \tx_h^m \cdot \underline D_j \tx_h^m + \mu_i \mu_j, \quad i,j=1,2,3, 
$$
where $\underline D_i$ here refers to the $i$--th component of the tangential gradient on $\mathcal M_h$. We 
used  \eqref{eq:weakfulldisc_simple} with $\alpha =1.0$ and $\alpha=0.04$ to solve the simulation of the evolution of 
a dumbbell shaped surface given by \eqref{db}. The results it produced where visually identical to those displayed in Figures \ref{evol_db} and \ref{evol_db2}.

\begin{appendix}
\setcounter{equation}{0} 
\renewcommand{\theequation}{\Alph{section}.\arabic{equation}}
\section{Appendix} \label{sec:AppA}

\noindent
{\it Proof of Lemma \ref{lem:Gprop}:} Let $\varphi:U \rightarrow \mathbb R^3$ be a local parametrization of $\mathcal M$ with $\mu=\varphi_{\theta_1} \wedge \varphi_{\theta_2}$.  Abbreviating
$X=x \circ \varphi$ we deduce with the help of \eqref{eq:tanggrad} that
\begin{eqnarray}
v_j \circ \varphi  & = &( \underline D_{j+1} x\wedge \underline D_{j+2}x) \circ \varphi = h^{rs} h^{pq} \varphi^{j+1}_{\theta_r} \varphi^{j+2}_{\theta_p} X_{\theta_s} \wedge X_{\theta_q} \nonumber \\
& = & \bigl( h^{r1} h^{p2} - h^{r2} h^{p1} \bigr)  \varphi^{j+1}_{\theta_r} \varphi^{j+2}_{\theta_p} X_{\theta_1} \wedge X_{\theta_2}  = \frac{1}{\mbox{det}(h_{ij})} \bigl( \varphi^{j+1}_{\theta_1} \varphi^{j+2}_{\theta_2} - \varphi^{j+1}_{\theta_2} \varphi^{j+2}_{\theta_1} \bigr) X_{\theta_1} \wedge X_{\theta_2}, \nonumber
\end{eqnarray}
since
\begin{displaymath}
 h^{r1}  h^{p2}- h^{r2} h^{p1}=
 \left\{
 \begin{array}{cl}
 0, & r=p, \\
 \displaystyle \frac{1}{\mbox{det}(h_{ij})}, & r=1, p=2, \\
\displaystyle  - \frac{1}{\mbox{det}(h_{ij})}, & r=2,p=1.
 \end{array}
 \right.
 \end{displaymath}
Observing that 
\begin{displaymath}
\varphi^{j+1}_{\theta_1} \varphi^{j+2}_{\theta_2} - \varphi^{j+1}_{\theta_2} \varphi^{j+2}_{\theta_1} = (\varphi_{\theta_1} \wedge \varphi_{\theta_2})_j = \sqrt{\mbox{det}(h_{ij})} \, \mu_j \circ \varphi, \quad j=1,2,3
\end{displaymath}
we therefore obtain recalling that  $ | X_{\theta_1} \wedge X_{\theta_2} |= \sqrt{\mbox{det}(g_{ij})}$ and \eqref{eq:detG}
\begin{displaymath}
v_j \circ \varphi = \frac{\sqrt{\mbox{det}(g_{ij})}}{\sqrt{\mbox{det}(h_{ij})}} \mu_j \circ \varphi \frac{X_{\theta_1} \wedge X_{\theta_2}}{| X_{\theta_1} \wedge X_{\theta_2} |} = \sqrt{\dg \circ \varphi} \,  \mu_j \circ \varphi  \frac{X_{\theta_1} \wedge X_{\theta_2}}{| X_{\theta_1} \wedge X_{\theta_2} |}, \quad j=1,2,3.
\end{displaymath}
If we multiply by $\mu_j \circ \varphi$ and sum over $j=1,2,3$ taking into account that $| \mu |=1$  we obtain
\begin{equation} \label{eq:nuform}
\nu \circ \varphi  = (\frac{1}{\sqrt{\dg}} \mu_j  v_j) \circ \varphi = \frac{X_{\theta_1} \wedge X_{\theta_2}}{| X_{\theta_1} \wedge X_{\theta_2} |},
\end{equation}
which shows that $\nu$ is a unit normal field to $\Gamma(t)$. Next, using \eqref{eq:Gijinv},  \eqref{eq:tanggrad} and arguing in a similar way as above we find  for fixed $i,\ell \in \lbrace 1,2,3 \rbrace$ recalling also that $(\mu_k \circ \varphi) \varphi^k_{\theta_{p'}}=(\mu_j  \circ \varphi) \varphi^j_{\theta_p}=0$
 \begin{eqnarray*}
 \lefteqn{
 [G^{ik} G^{j \ell} (\underline D_j x \wedge \underline D_k x)] \circ \varphi } \\
 & = & (g^{rs} \varphi^{i}_{\theta_r} \varphi^k_{\theta_s} + (\mu_i \mu_k)\circ \varphi) (g^{mn} \varphi^j_{\theta_m} \varphi^\ell_{\theta_n} +(\mu_j \mu_\ell) \circ \varphi) \bigl( h^{pq} \varphi^j_{\theta_p} X_{\theta_q} 
 \wedge h^{p'q'} \varphi^k_{\theta_{p'}} X_{\theta_{q'}} \bigr) \\
 & = & g^{rs} g^{mn} h_{mp} h_{s p'} \varphi^{i}_{\theta_r}\varphi^\ell_{\theta_n} h^{pq}h^{p'q'} (X_{\theta_q} \wedge X_{\theta_{q'}})  
  =  g^{rs} g^{mn} \varphi^{i}_{\theta_r}\varphi^\ell_{\theta_n} (X_{\theta_m} \wedge X_{\theta_s})  \\
 & = & (g^{1n}  g^{r2}- g^{2n} g^{1r})  \varphi^{i}_{\theta_r}\varphi^\ell_{\theta_n} (X_{\theta_1} \wedge X_{\theta_2}) 
    = (\varphi^{i}_{\theta_2}\varphi^\ell_{\theta_1}-  \varphi^{i}_{\theta_1}\varphi^\ell_{\theta_2})  \frac{1}{\mbox{det}(g_{ij})} (X_{\theta_1} \wedge X_{\theta_2}).
   \end{eqnarray*}
  Observing that 
 \begin{displaymath}
  \varphi^{i}_{\theta_2}\varphi^\ell_{\theta_1}-  \varphi^{i}_{\theta_1}\varphi^\ell_{\theta_2}  =(-1)^{i+\ell}
 ( \varphi_{\theta_1} \wedge \varphi_{\theta_2})_{\sigma(i,\ell)}=     (-1)^{i+\ell} \sqrt{\mbox{det}(h_{ij})}   \mu_{\sigma(i,\ell)} \circ \varphi, \quad 1 \leq i< \ell \leq 3
  \end{displaymath}
  we therefore obtain with the help of \eqref{eq:detG} and \eqref{eq:nuform}
\begin{displaymath}
[G^{ik} G^{j \ell} (\underline D_j x \wedge \underline D_k x)] \circ \varphi = \frac{(-1)^{i+\ell}  \mu_{\sigma(i,\ell)}  }{\sqrt{\dg}}  \circ \varphi \frac{X_{\theta_1} \wedge X_{\theta_2}}{| X_{\theta_1} \wedge X_{\theta_2} |} = (\frac{1}{\sqrt{\dg}}   (-1)^{i+\ell}  \mu_{\sigma(i,\ell)} \nu) \circ \varphi
\end{displaymath}
which implies \eqref{eq:form1}. \\[2mm]

\end{appendix}

\bibliographystyle{plain}


\end{document}